\documentclass[11pt,a4paper]{article}%
\usepackage{amssymb,amsmath,amsfonts,amsthm,array,bm,color}%
\usepackage{amsmath}%
\usepackage{amsfonts}
\usepackage{latexsym}
\usepackage{amssymb}%
\usepackage{graphicx}
\usepackage{hyperref}
\hypersetup{colorlinks=true, linkcolor=red, anchorcolor=blue, citecolor=blue}

\providecommand{\U}[1]{\protect\rule{.1in}{.1in}}
\numberwithin{equation}{section}
\newtheorem{theorem}{Theorem}[section]
\newtheorem{lemma}[theorem]{Lemma}

\newtheorem{remark}[theorem]{Remark}

\newtheorem{definition}[theorem]{Definition}

\def\<{\langle}
\def\>{\rangle}
\def\d{{\rm d}}

\def\P{\mathbb{P}}
\def\R{\mathbb{R}}
\def\T{\mathbb{T}}
\def\Z{\mathbb{Z}}

\begin{document}

\title{Scaling Limits for Stochastic SQG Equations and 2D Inviscid Critical Boussinesq Equations with Transport Noises}

\author{Shuchen Guo\footnote{Email: guoshuchen15@mails.ucas.ac.cn. School of Mathematical Sciences, University of the Chinese Academy of Sciences, Beijing, China and Academy of Mathematics and Systems Science, Chinese Academy of Sciences, Beijing, China.} }

\date{}

\maketitle
\begin{abstract}
We study stochastic SQG equations on the torus $\mathbb{T}^2$ with  multiplicative transport noise in the $L^2$-setting. Under a suitable scaling of the noise, we show that the solutions converge weakly to the unique solution to the deterministic dissipative SQG equation. A similar scaling limit result is proved also for the stochastic 2D inviscid critical Boussinesq equations.
\end{abstract}

\section{Introduction}
\hspace{1.5em}Fluid mechanical equations with multiplicative noises of transport type have been considered a long time ago, see e.g. \cite{BCF} and also the paper \cite{MikRoz}. Such noises emerge naturally from stochastic model reduction techniques \cite{MajdaTV, AFP} and are also stressed by variational considerations \cite{Holm, DH}. In the past ten years, there is an increasing interest in considering fluid equations perturbed by transport noises, partly motivated by the theory of regularization by noise, cf. \cite{FGP, FGP2} and more recent papers \cite{BFM, FL-19, BM-19, CL-19} on the vorticity form of stochastic 2D Euler equations with transport noises, showing existence of solutions in various regularity classes. In particular, using the method of point vortex approximation, existence of white noise solutions was proved in \cite{FL-19} when the noise has finite trace. In the case of a noise with infinite trace, it was shown in \cite{FlaLuo-20} that, under a suitable scaling of the noise, stationary white noise solutions to a sequence of stochastic 2D Euler equations converge weakly to the unique stationary solution to the 2D Navier-Stokes equations driven by space-time white noise, a model which has been studied many years ago, see e.g. \cite{AC, AF, DaPD}. Remark that the original equation is hyperbolic in nature, while the limit equation is dissipative; furthermore, the uniqueness of solutions to the approximating equations remains an open problem, but it holds true for the limit equation, and thus it is reasonable to say that transport noise asymptotically regularizes 2D Euler equations.

It turns out that such scaling limit is more relevant to solutions in $L^2$-spaces. For instance, linear transport equation perturbed by transport noise was considered in \cite{Gal}, where the limit equation is a second order deterministic parabolic equation. Later on, similar results are proved for stochastic 2D Euler, mSQG and inviscid Boussinesq equations with transport noises, see \cite{FGL, LS, Luo20}. This idea has also been applied to some dissipative systems, showing the phenomenon of dissipation enhancement, see \cite{FL-19b} for the vorticity form of 3D Navier-Stokes equations and \cite{FGL20} for some general nonlinear equations whose solutions might explode for initial data above some threshold.

Inspired by the above-mentioned works, we consider in the present paper the scaling limits for stochastic SQG equations and 2D critical inviscid Boussinesq equations. We remark that, since stream functions in these models have the same regularity as the velocity fields, the proof of convergence of nonlinear terms is more difficult and we shall appeal to a commutator estimate as in \cite{Resnick}.

Before going to the details, we introduce a few notations that will be used below. Let $\T^2= \R^2/\Z^2$ be the 2D torus and $\Z^2_0$ the nonzero lattice points; $\ell^2= \ell^2(\Z^2_0)$ is the usual space of square summable real sequences indexed by $\Z^2_0$.  $\{\sigma_k\}_{k\in\Z^2_0}$ is a CONS of the space of divergence-free vector fields on $\T^2$ with zero mean (see Section 2 for a precise choice); $L^2(\T^2)$ and Sobolev space $H^s(\T^2)$ $(s\in\R)$ are assumed to consists of functions with zero mean. Finally, $\{B^k\}_{k\in \Z^2_0}$ is a family of independent standard Brownian motions defined on some filtered probability space $(\Omega, \mathcal F, (\mathcal F_t)_{t\geq 0}, \P)$.

\subsection{Stochastic SQG Equations}\label{subs-intro-sqg}
\hspace{1.5em}The deterministic SQG equation describes the temperature in a rapidly rotating stratified fluid with uniform potential vorticity, with applications in oceanic and meteorologic flows. This equation is important in mathematics due to its structural similarity with the 3D Euler equations, see \cite[Section 2]{CMT}. We refer to the introduction of \cite{FS} for a quite complete account of well posedness results for the SQG equation.  On the 2D torus $\T^2$, it reads as
  $$\left\{\aligned & \partial_t \omega + u\cdot\nabla\omega =0, \\
  &u = \nabla^\perp (-\Lambda^{-1} \omega),
  \endaligned \right.$$
where $\nabla^\perp=(\partial_2, -\partial_1)$ and $\partial_i = \partial_{x_i}$; $\Lambda =(-\Delta)^{1/2}$ and $\Delta$ is the periodic Laplacian on $\T^2$. To simplify the notation, we write $u = K_0 \ast \omega$ where $K_0$ is the kernel corresponding to the operator $-\nabla^\perp \Lambda^{-1}$; it is known that $K_0$ has the asymptotic behavior
  $$K_0(x) \sim \frac{x^\perp}{|x|^{3}} \quad \mbox{for } |x|\to 0.$$

We take a $\theta\in \ell^2$ which is symmetric, i.e.
  \begin{equation}\label{symmetry}
  \theta_k = \theta_l \quad \mbox{for all } k,l\in\Z^2_0 \mbox{ with } |k|=|l|.
  \end{equation}
Consider the following stochastic SQG equation with transport noise:
  \begin{align*}
  	d\omega + u\cdot\nabla\omega \,\d t = \frac{2\sqrt{\nu}}{\|\theta \|_{\ell^2}} \sum_{k} \theta_k \sigma_k \cdot\nabla\omega\circ \d B^k_t,
  \end{align*}
where $u= K_0 \ast \omega$ and $\nu>0$ stands for the noise intensity. Using relation \eqref{tensor} below, one can show that this equation has the It\^o formulation:
  \begin{align}\label{ito}
  	\d\omega + u\cdot\nabla\omega \,\d t = \nu \Delta \omega\,\d t + \frac{2\sqrt{\nu}}{\|\theta \|_{\ell^2}} \sum_{k} \theta_k \sigma_k \cdot\nabla\omega \d B^k_t.
  \end{align}
For $\omega_0\in L^2(\T^2)$, it is not difficult to show the existence of a weak solution $\omega$ with trajectories in $L^\infty(0,T; L^2(\T^2))$, see Theorem \ref{main}.

Next, we take a sequence $\{\theta^N\}_{N\geq 1} \subset \ell^2$, each satisfying the symmetry property \eqref{symmetry}, and consider the following stochastic SQG equations:
  \begin{equation}\label{SSQG}
  \d\omega^N + u^N\cdot\nabla\omega^N \,\d t = \frac{2\sqrt{\nu}}{\|\theta^N \|_{\ell^2}} \sum_{k} \theta^N_k \sigma_k \cdot\nabla\omega^N\circ \d B^k_t,
  \end{equation}
where $u^N= K_0 \ast \omega^N$. These equations are also understood as above in the It\^o sense. In order to prove a scaling limit for these solutions, we need the following condition:
  \begin{equation}\label{theta-N}
  \lim_{N\to \infty} \frac{\|\theta^N \|_{\ell^\infty}}{\|\theta^N \|_{\ell^2}} =0,
  \end{equation}
which holds for many different choices of $\{\theta^N\}_{N\geq 1}$. Our first main result is:
\begin{theorem} \label{scalingtheoremSQG}
	Assume (\ref{theta-N}) holds and $\{\omega_0^N\}_{N\geq1}$ converges weakly to some $\omega_0\in L^2(\mathbb{T}^2)$.
	Then, the weak solutions to equation (\ref{SSQG}) converge weakly as $N\rightarrow\infty$ to the unique weak solution to the dissipative SQG equation:
	\begin{equation}\label{dissipativeSQG}
		\left\{\begin{array}{l}
			\partial_{t} \omega+u \cdot \nabla \omega=\nu \Delta \omega, \\
			u=K_{0} * \omega, \quad \omega_{0} \in L^{2}\left(\mathbb{T}^{2}\right).
		\end{array}\right.
	\end{equation}
\end{theorem}
We remark that so far the uniqueness of the deterministic SQG equations remains an open problem. The above theorem tells us that stochastic SQG equations \eqref{SSQG} with transport noise converge in a scaling limit to \eqref{dissipativeSQG} which is uniquely solvable. We shall slightly strengthen this assertion in Theorem \ref{SQGconvergenceinProb} below. Consequently, one can show that the distance between the probability laws of any two weak solutions to \eqref{SSQG} vanishes as $N$ to $\infty$. In this sense, it is reasonable to say that weak solutions to stochastic SQG equations \eqref{SSQG} are approximately unique; see Section 6.1 in \cite{FGL} for more detailed discussions in the setting of stochastic Euler equations.

\subsection{Stochastic 2D Critical Boussinesq Equations}
\hspace{1.5em}Our next result is concerned with the 2D inviscid Boussinesq equations describing the evolution of an incompressible fluid, subject to a vertical force which is proportional to some scalar field, such as the temperature. This part is motivated by \cite{Luo20}, where one can find some references on various well posedness results on the equations. In the current paper, we consider the critical case as in Subsection \ref{subs-intro-sqg}. Written in vorticity form, the deterministic Boussinesq equations with thermal diffusion on torus $\T^2$ reads as
$$\left\{  \aligned
&\partial_t \xi + u\cdot \nabla\xi = \kappa \Delta\xi, \\
&\partial_t \omega + u\cdot \nabla\omega = \partial_1 \xi,
\endaligned \right.$$
where $\xi$ is a scalar field representing the temperature, $u$ and $\omega$ are the velocity and vorticity fields of the fluid, related to one another by $u= K_0\ast \omega$ with the kernel $K_0$ as in the last subsection, $\kappa>0$ is the thermal diffusivity.

We take a $\theta\in\ell^2$ which fulfills \eqref{symmetry}, and consider the following stochastic critical Boussinesq equations with transport noise (fix $\nu>0$):
\begin{equation*}
	\left\{\begin{array}{l}   \displaystyle
		\mathrm{d} \xi+u \cdot \nabla \xi \mathrm{d} t=\kappa \Delta \xi \mathrm{d} t+\frac{2 \sqrt{\nu}}{\|\theta\|_{\ell^{2}}} \sum_{k} \theta_{k}\sigma_{k}\cdot \nabla \xi \circ \mathrm{d} B_{t}^{k}, \\
		\displaystyle \d \omega+u \cdot \nabla \omega \mathrm{d} t=\partial_{1} \xi \mathrm{d} t+\frac{2 \sqrt{\nu}}{\|\theta\|_{\ell^{2}}} \sum_{k} \theta_{k}\sigma_{k}  \cdot \nabla \omega \circ \mathrm{d} B_{t}^{k},
	\end{array}\right.
\end{equation*}
where $u=K_0 * \omega$ and $(\xi_{0},\omega_{0}) \in \left(L^{2}(\mathbb{T}^{2})\right)^2$. In the same way as stochastic SQG equation, it also can be reformulated into It\^{o} form by using \eqref{tensor} below:
\begin{equation}\label{itoboussinesq}
	\left\{\begin{array}{l}
		\displaystyle	\mathrm{d} \xi+u \cdot \nabla \xi \mathrm{d} t=(\kappa+\nu) \Delta \xi \mathrm{d} t+\frac{2 \sqrt{\nu}}{\left\|\theta\right\|_{\ell^{2}}}\sum_{k} \theta_{k}\sigma_{k}\cdot \nabla \xi \mathrm{d} B_{t}^{k}, \\
		\displaystyle  \d \omega+u \cdot \nabla \omega \mathrm{d} t=(\partial_{1} \xi +\nu\Delta\omega)\mathrm{d} t+\frac{2 \sqrt{\nu}}{\left\|\theta\right\|_{\ell^{2}}} \sum_{k} \theta_{k}\sigma_{k}  \cdot \nabla \omega \mathrm{d} B_{t}^{k}.
	\end{array}\right.
\end{equation} 
The existence of a weak solutions to (\ref{itoboussinesq}) will be shown in Theorem \ref{existenceBoussinesq}.

Take $\{\theta^N\}_{N\geq 1} \subset \ell^2$, each $\theta^N$ satisfying \eqref{symmetry}, we consider a family of stochastic critical Boussinesq equations:
\begin{equation}\label{stratonovichboussinesq}  
	\left\{\begin{array}{l}   \displaystyle
		\mathrm{d} \xi^N+u^N \cdot \nabla \xi^N \mathrm{d} t=\kappa \Delta \xi^N \mathrm{d} t+\frac{2 \sqrt{\nu}}{\|\theta^N\|_{\ell^{2}}} \sum_{k} \theta^N_{k}\sigma_{k}\cdot \nabla \xi^N \circ \mathrm{d} B_{t}^{k}, \\
		\displaystyle \d \omega^N+u^N \cdot \nabla \omega^N \mathrm{d} t=\partial_{1} \xi^N \mathrm{d} t+\frac{2 \sqrt{\nu}}{\|\theta^N\|_{\ell^{2}}} \sum_{k} \theta^N_{k}\sigma_{k}  \cdot \nabla \omega^N \circ \mathrm{d} B_{t}^{k},
	\end{array}\right.
\end{equation}
where $u^N=K_0 * \omega^N$ and $(\xi^N_{0},\omega^N_{0}) \in \left(L^{2}(\mathbb{T}^{2})\right)^2$. 
 Our second main result is:

\begin{theorem}\label{scalingresultforBoussinesq}
	Assume (\ref{theta-N}) and $\left(\xi_{0}^{N},\omega_{0}^N\right)$ converge weakly to some $\left(\xi_{0},\omega_{0}\right)\in \left(L^{2}(\mathbb{T}^{2})\right)^2$. Then, the weak solutions to (\ref{stratonovichboussinesq}) converges weakly as $N\rightarrow\infty$ to the unique weak
	solution to the deterministic viscous Boussinesq equations:
	\begin{equation}\label{deterministicviscousBoussinesq}
		\left\{\begin{array}{l}
			\partial_{t} \xi+u \cdot \nabla \xi=(\kappa+\nu) \Delta \xi, \\
			\partial_{t} \omega+u \cdot \nabla \omega=\partial_{1} \xi+\nu \Delta \omega,\\
			u=K_0*\omega,
		\end{array}\right.
	\end{equation}
	with initial data $\left(\xi_{0}, \omega_{0}\right)$. 
\end{theorem}
Although the uniqueness of the deterministic critical Boussinesq equations is still open, the above theorem shows that stochastic critical Boussinesq equations \eqref{stratonovichboussinesq} with transport noise converge to a uniquely solvable equation \eqref{deterministicviscousBoussinesq} in a scaling limit.  
We explain this assertion further by Theorem \ref{BoussinesqconvergenceinProb}, which deducing that the distance between laws of any two weak solutions to \eqref{stratonovichboussinesq} tends to zero as $N$ to $\infty$. Similar to the stochastic SQG equations, we say that weak solutions to stochastic Boussinesq equations \eqref{stratonovichboussinesq} are approximately unique.

The paper is organized as follows. In section 2, we introduce some notations and lemmas. We first prove the existence of solution to stochastic SQG equation in subection 3.1; and then we prove the scaling limit result (Theorem \ref{scalingtheoremSQG}) in subsection 3.2; the last subsection is devoted to proving the uniqueness of the limit equation \eqref{dissipativeSQG}.  Similarly to section 3, we prove in section 4 the scaling limit result for stochastic Boussinesq equations (Theorem \ref{scalingresultforBoussinesq}).

\section{Preliminaries}
\hspace{1.5em}Here we summarize some definitions and lemmas which appear many times in the main text.
Firstly, we define a CONS $\left\{\sigma_{k}\right\}_{k \in \mathbb{Z}_{0}^{2}}$ of the space of square integrable and divergence-free vector fields on torus $\mathbb{T}^{2}$ as
\begin{align}\label{sigma}
	\sigma_{k}(x)=\frac{k^{\perp}}{|k|} e_{k}(x), \quad x \in \mathbb{T}^{2}, \quad k \in \mathbb{Z}_{0}^{2}:=\mathbb{Z}^{2}\backslash\{0\},
\end{align}
where $k^{\perp}=(k_2,-k_1)$ and
\begin{align}\label{e}
	e_{k}(x)=\sqrt{2}\left\{\begin{array}{ll}
		\cos (2 \pi k \cdot x), & k \in \mathbb{Z}_{+}^{2}, \\
		\sin (2 \pi k \cdot x), & k \in \mathbb{Z}_{-}^{2}.
	\end{array}\right.
\end{align}
The gradient and Laplacian of $e_k$ are as followed,
\begin{align}\label{factofbasis}
	\nabla e_{k}=2 \pi k e_{-k},\quad \Delta e_{k}=-4 \pi^{2}|k|^{2} e_{k}.
\end{align}
For any sequence $\theta:=\left\{\theta_{k}\right\}_{k \in \mathbb{Z}_{0}^{2}} \in \ell^{2}$ satisfying (\ref{symmetry}), the following identity (see Lemma 2.6 in \cite{FlaLuo-20}) holds:
\begin{align}\label{tensor}
	\sum_{k \in \mathbb{Z}_{0}^{2}} \theta_{k}^{2} \sigma_{k}(x) \otimes \sigma_{k}(x) \equiv \frac{1}{2}\|\theta\|_{\ell^{2}}^{2} I_{2}, \quad x \in \mathbb{T}^{2},
\end{align}
where $I_2$ is the two dimensional identity matrix.

Next, we recall some compactness results proved by J.Simon in \cite{simon1986compact}. The assertion (i) below follows Corollary 9 in his paper, while assertion (ii) is a direct consequence of Corollary 5.
\begin{theorem}\label{embedding}
	(i)	For any $\delta\in(0,1)$, let $\beta>5$, if $p>12(\beta-\delta)$, then
	\begin{align}
		L^{p}(0, T ; L^2) \cap W^{1 / 3,4}\left(0, T ; H^{-\beta}\right) \subset C\left([0, T] ; H^{-\delta}\right)
	\end{align}
	is a compact embedding.
	\\	(ii) For any $\gamma\in(0,1/2)$, let $\beta>5$, then
	\begin{equation}L^{2}\left(0, T ; H^{1}\right) \cap W^{\gamma, 2}\left(0, T ; H^{-\beta}\right) \subset L^{2}\left(0, T ; L^{2}\right)
	\end{equation}
	is a compact embedding.
\end{theorem}
Finally, the square root of the Laplacian on $\T^2$ is defined as $\Lambda=(-\Delta)^{1 / 2}$. By \cite{Resnick} and \cite{roncal2016fractional}, for a integrable function $f:\mathbb{T}^{2}\rightarrow\mathbb{R}$, $\Lambda f(x)$ can be expressed by the multiple Fourier series expansion as
\begin{align}
	\Lambda f(x)=F^{-1}\left[F[\Lambda f]\right](x)=(2\pi)^{-2}\sum_{k\in\mathbb{Z}^{2}} |k|\int_{\mathbb{T}^{2}} f(y)e^{-iy\cdot k}\mathrm{d}y e^{ix\cdot k}.
\end{align}
With $\phi$  regular enough and $\psi\in H^1(\mathbb{T}^{2})$, we define the commutator as
\begin{align}\label{commutator}
	[\Lambda, \nabla \phi] \psi=\Lambda(\psi\nabla \phi )-(\Lambda \psi)\nabla \phi ,
\end{align}
and the following estimate valid for any $\epsilon > 0$:
\begin{align}\label{commutatorestimate}
	\|[\Lambda, \nabla \phi] \psi\|_{L^2(\mathbb{T}^{2},\mathbb{R}^{2})}  \leqslant C_{\epsilon}\|\phi\|_{H^{3+\epsilon}(\mathbb{T}^{2})}\|\psi\|_{L^2(\mathbb{T}^{2})}.
\end{align}
Chapter 3.6 of \cite{taylor2012pseudodifferential} is devoted to prove above commutator estimate, where we can find more details of that.
\section{The Scaling limit of Stochastic SQG Equations}
\hspace{1.5em}This section consists of three parts: In subsection 3.1, we prove the existence of solution to stochastic SQG equation \eqref{ito}; then we give the proof of main scaling theorem of SQG equation (Theorem \ref{scalingtheoremSQG}) and some remarks in subsection 3.2; the last subsection 3.3 presents the proof of the uniqueness of solution to deterministic dissipative SQG equation.
\subsection{Existence of Weak Solutions}
\hspace{1.5em} We prove the existence using the Galerkin approximation. Before proceeding further, we give precise definition of weak solutions to \eqref{ito}:
\begin{definition}\label{weaksolutionofSQG}
	We say that equation (\ref{ito}) has a weak solution if there exists a filtered probability space $\left(\Omega, \mathcal{F}, \mathcal{F}_{t}, \mathbb{P}\right)$, a sequence of independent $\mathcal{F}_{t}$-Brownian motions $\left\{B^{k}\right\}_{k \in \mathbb{Z}_{0}^{2}}$ and an $\mathcal{F}_{t}$-progressively measurable process $\omega \in L^{2}\left(\Omega, L^{2}\left(0, T ; L^{2}\right)\right)$ with $\mathbb{P}$-a.s. weakly continuous trajectories such that for any $\phi \in C^{\infty}\left(\mathbb{T}^{2}\right),$ the following equality holds $\mathbb{P}$-a.s. for all $t \in[0, T]$,
	\begin{equation}\label{solution}
		\begin{aligned}
			\left\langle\omega(t), \phi\right\rangle=&\left\langle\omega_{0}, \phi\right\rangle+\int_{0}^{t}\left\langle\omega(s), u(s) \cdot \nabla \phi\right\rangle \mathrm{d} s+\nu \int_{0}^{t}\left\langle\omega(s), \Delta \phi\right\rangle \mathrm{d} s \\
			&-\frac{2 \sqrt{\nu}}{\|\theta\|_{\ell^{2}}} \sum_{k \in \mathbb{Z}_{0}^{2}} \theta_{k} \int_{0}^{t}\left\langle\omega(s), \sigma_{k} \cdot \nabla \phi\right\rangle \mathrm{d} B_{s}^{k}.
		\end{aligned}
	\end{equation}
\end{definition}
\begin{remark}\label{isometry}
	Assume $\omega$ is an $\mathcal{F}_{t}$-progressively measurable process in $L^{2}\left(\Omega, L^{2}\left(0, T ; L^{2}\right)\right)$. Due to $\{\sigma_{k}\}_{k \in \mathbb{Z}_{0}^{2}}$ forming an incomplete orthonormal family in $L^2\left(\mathbb{T}^{2},\mathbb{R}^{2}\right)$, we have
	\begin{equation*}\begin{aligned}
			\mathbb{E}\left(\sum_{k \in \mathbb{Z}_{0}^{2}} \theta_{k}^{2} \int_{0}^{t}\left\langle\omega(s), \sigma_{k} \cdot \nabla \phi\right\rangle^{2} \mathrm{d} s\right)
			& \leq\|\theta\|_{\ell^{\infty}}^{2} \mathbb{E} \int_{0}^{T} \sum_{k \in \mathbb{Z}_{0}^{2}}\left\langle\omega(s) \nabla \phi, \sigma_{k}\right\rangle^{2} \mathrm{d} s \\
			& \leq\|\theta\|_{\ell^{\infty}}^{2} \mathbb{E} \int_{0}^{T}\left\|\omega(s) \nabla \phi\right\|_{L^{2}}^{2} \mathrm{d} s \\
			& \leq\|\theta\|_{\ell^{\infty}}^{2}\|\nabla \phi\|_{\infty}^{2} \mathbb{E} \int_{0}^{T}\left\|\omega(s)\right\|_{L^{2}}^{2} \mathrm{d} s \\
			&<+\infty.
	\end{aligned}\end{equation*}
	Thus the stochastic integral part of (\ref{solution}) makes sense.
\end{remark}
The following theorem is the existence results of weak solutions to \eqref{ito}.
\begin{theorem}\label{main}
	For any $\omega_{0} \in L^{2},$ there exists a weak solution satisfying
	\begin{align*}
		\sup _{t \in[0, T]}\left\|\omega(t)\right\|_{L^{2}} \leq\left\|\omega_{0}\right\|_{L^{2}} \quad \mathbb{P}\text{-a.s.}
	\end{align*}
\end{theorem}
First, we will look for an approximate solution $\omega_N(\cdot)$. For $N \geq 1,$ let $H_{N}=\operatorname{span}\left\{e_{k}:k \in \mathbb{Z}_{0}^{2},|k| \leq N\right\}$,
which is a finite-dimensional subspace of $L^{2}\left(\mathbb{T}^{2}\right)$. Denote the orthogonal projection $\Pi_{N}: L^{2}\left(\mathbb{T}^{2}\right) \rightarrow H_{N}$ and $\omega_N:=\Pi_{N} \omega=\sum_{|k| \leq N}\left\langle\omega, e_{k}\right\rangle e_{k},$ 

Let
\begin{equation}\label{bN}
	\begin{aligned}
		b_{N}(\omega) &=\Pi_{N}\left(\left(K_{0} * \Pi_{N} \omega\right) \cdot \nabla\left(\Pi_{N} \omega\right)\right), \\
		G_{N}^{k}(\omega) &=\Pi_{N}\left(\sigma_{k} \cdot \nabla\left(\Pi_{N} \omega\right)\right), \quad \omega \in L^{2}, \quad k \in \mathbb{Z}_{0}^{2}.
	\end{aligned}
\end{equation}
Note that, for fixed $N,$ there are only finitely many $k \in \mathbb{Z}_{0}^{2}$ such that $G_{N}^{k}$ is not zero. We shall view $b_{N}$ and $G_{N}^{k}$ as vector fields on $H_{N}$ whose generic element is denoted by $\omega_{N}$. The following useful properties hold:
\begin{align}\label{property}
	\left\langle b_{N}\left(\omega_{N}\right), \omega_{N}\right\rangle=\left\langle G_{N}^{k}\left(\omega_{N}\right), \omega_{N}\right\rangle=0, \quad \text { for all } \omega_{N} \in H_{N}.
\end{align}
Consider the finite dimensional version of (\ref{ito}) on $H_N$:
\begin{equation}\label{finitedimension}
	\left\{\begin{array}{l}
		\mathrm{d} \omega_N(t)
		=-b_{N}\left(\omega_{N}(t)\right) \mathrm{d} t+\nu \Delta \omega_{N}(t) \mathrm{d} t+\frac{2 \sqrt{\nu}}{\|\theta\|_{\ell^{2}}} \sum_{k \in \mathbb{Z}_{0}^{2}} \theta_{k} G_{N}^{k}\left(\omega_{N}(t)\right) \mathrm{d} B_{t}^{k},\\
		\omega_N(0)=\Pi_N\omega_0,
	\end{array}\right.
\end{equation}
with $\omega_0\in L^2$ is the initial condition in Theorem \ref{main}. It follows from standard SDE theory that there exists a unique local strong solution $\omega_N(\cdot)$ under any initial conditions.

Next, we will obtain a priori estimate of $\omega_N$. By It\^{o} formula,
\begin{equation*}
	\begin{aligned}
		\mathrm{d}\left\|\omega_{N}(t)\right\|_{L^{2}}^{2}=&-2\left\langle\omega_{N}(t), b_{N}\left(\omega_{N}(t)\right)\right\rangle \mathrm{d} t+2 \nu\left\langle\omega_{N}(t), \Delta \omega_{N}(t)\right\rangle \mathrm{d} t \\
		&+ \frac{4 \sqrt{\nu}}{\|\theta\|_{\ell^{2}}} \sum_{k \in \mathbb{Z}_{0}^{2}} \theta_{k}\left\langle\omega_{N}(t), G_{N}^{k}\left(\omega_{N}(t)\right)\right\rangle \mathrm{d} B_{t}^{k} \\
		&+\frac{4 \nu}{\|\theta\|^2_{\ell^{2}}} \sum_{k \in \mathbb{Z}_{0}^{2}} \theta_{k}^{2}\left\|G_{N}^{k}\left(\omega_{N}(t)\right)\right\|_{L^{2}}^{2} \mathrm{d} t.
	\end{aligned}
\end{equation*}
The first and the third terms on the right-hand side vanish due to properties (\ref{property}). And the following relation holds:
\begin{equation*}
	\begin{aligned}
		\left\|G_{N}^{k}\left(\omega_{N}(t)\right)\right\|_{L^{2}}
		\leq\left\|\sigma_{k} \cdot \nabla \omega_{N}(t)\right\|_{L^{2}}.
	\end{aligned}
\end{equation*}
Therefore,

\begin{align*}
	\frac{4 \nu}{\|\theta\|^2_{\ell^{2}}} \sum_{k \in \mathbb{Z}_{0}^{2}} \theta_{k}^{2}\left\|G_{N}^{k}\left(\omega_{N}(t)\right)\right\|_{L^{2}}^{2} & \leq \frac{4 \nu}{\|\theta\|^2_{\ell^{2}}} \sum_{k \in \mathbb{Z}_{0}^{2}} \theta_{k}^{2} \int_{\mathbb{T}^{2}}\left(\sigma_{k} \cdot \nabla \omega_{N}(t)\right)^{2} \mathrm{d} x \\
	&=2 \nu\left\|\nabla \omega_{N}(t)\right\|_{L^{2}}^{2},
\end{align*}
where the last equality is due to (\ref{tensor}).
Combining these results above, we obtain $\mathrm{d}\left\|\omega_{N}(t)\right\|_{L^{2}}^{2} \leq 0$, which implies the following inequality and thus the global existence of solution to equation (\ref{finitedimension}) holds; moreover,
\begin{equation}\label{finiteL2ineq}
	\sup _{t \in[0, T]}\left\|\omega_{N}(t)\right\|_{L^{2}} \leq\left\|\omega_{N}(0)\right\|_{L^{2}} \leq\left\|\omega_{0}\right\|_{L^{2}}.
\end{equation}
Because for any $p>1$:
\begin{align*}
	\mathbb{E} \int_{0}^{T}\left\|\omega_{N}(t)\right\|_{L^{2}}^{p} \mathrm{d} t \leq T\left\|\omega_{N}(0)\right\|_{L^{2}}^{p} \leq T\left\|\omega_{0}\right\|_{L^{2}}^{p},
\end{align*}
then $\{\omega_N(\cdot)\}_{N\geq1}$ is uniformly bounded in $L^{p}\left(\Omega, L^{p}\left(0, T ; L^{2}\right)\right)$. Since $u_N=K_0*\omega_{N}$, $\{u_N(\cdot)\}_{N\geq1}$ is also uniformly bounded in $L^{p}\left(\Omega, L^{p}\left(0, T ; L^{2}\right)\right)$.


Then, we will look for a candidate $(\tilde{\omega}(\cdot),B_{\cdot})$ as a weak solution to (\ref{solution}) applying the compactness results. Similar to the methods using in \cite{FGL} and \cite{LS},
let $Q_N$ denotes the law of $\omega_N$. By the assertion (i) of Theorem \ref{embedding}, to show $\{Q_N\}_{N\geq1}$ is tight on $C\left([0, T] ; H^{-\delta}\right)$, it is sufficient to prove
\begin{align}\label{boundedness}
	\sup _{N \geq 1}\mathbb{E} \int_{0}^{T}\left\|\omega_{N}(t)\right\|_{L^2}^{p} \mathrm{d} t+\sup _{N \geq 1} \mathbb{E} \int_{0}^{T} \int_{0}^{T} \frac{\left\|\omega_{N}(t)-\omega_{N}(s)\right\|_{H^{-\beta}}^{4}}{|t-s|^{7 / 3}} \mathrm{d} t \mathrm{d} s<\infty.
\end{align}
It remains to estimate the second term on the left-hand side of this inequality.
\begin{lemma}\label{lemma1}
	There is a constant $C=C(\nu,T,\left\|\omega_{0}\right\|_{L^2})$ such that for any $N \geq 1$ and $0 \leq s<t \leq T$,
	\begin{align}\label{four}
		\mathbb{E}\left(\left\langle\omega_{N}(t)-\omega_{N}(s), e_{k}\right\rangle^{4}\right) \leq C|k|^{8}|t-s|^{2}, \quad \text { for all } k \in \mathbb{Z}_{0}^{2}.
	\end{align}
	\begin{proof}
		It is enough to consider $|k| \leq N .$ By (\ref{finitedimension}) we obtain
		\begin{equation}\label{le1}
			\begin{aligned}
				\left\langle\omega_{N}(t)-\omega_{N}(s), e_{k}\right\rangle=& \int_{s}^{t}\left\langle\omega_{N}(r), u_{N}(r) \cdot \nabla e_{k}\right\rangle \mathrm{d} r+\nu \int_{s}^{t}\left\langle\omega_{N}(r), \Delta e_{k}\right\rangle \mathrm{d} r \\
				&-\frac{2 \sqrt{\nu}}{\|\theta\|_{\ell^{2}}} \sum_{l \in \mathbb{Z}_{0}^{2}} \theta_{l} \int_{s}^{t}\left\langle\omega_{N}(r), \sigma_{l} \cdot \nabla e_{k}\right\rangle \mathrm{d} B_{r}^{l}.
			\end{aligned}
		\end{equation}
		Firstly, using H\"{o}lder inequality and inequality (\ref{finiteL2ineq}), we have
		\begin{equation*}
			\begin{aligned}
				\mathbb{E}\left(\left|\int_{s}^{t}\left\langle\omega_{N}(r), u_{N}(r) \cdot \nabla e_{k}\right\rangle \mathrm{d} r\right|^{4}\right) & \leq|t-s|^{3} \mathbb{E} \int_{s}^{t}\left\langle\omega_{N}(r), u_{N}(r) \cdot \nabla e_{k}\right\rangle^{4} \mathrm{d} r \\
				& \leq|t-s|^{3} \mathbb{E} \int_{s}^{t}\left\|\omega_{N}(r)\right\|_{L^{2}}^{4}\left\|u_{N}(r)\right\|_{L^{2}}^{4}\left\|\nabla e_{k}\right\|_{\infty}^{4} \mathrm{d} r \\
				& \leq C\left\|\omega_{0}\right\|_{L^{2}}^{8}|k|^{4}|t-s|^{4},
			\end{aligned}
		\end{equation*}
		where the last step is due to the fact (\ref{factofbasis}). Similarly, the estimate
		\begin{equation*}
			\mathbb{E}\left(\left|\int_{s}^{t}\left\langle\omega_{N}(r), \Delta e_{k}\right\rangle \mathrm{d} r\right|^{4}\right) \leq C \left\|\omega_{0}\right\|_{L^{2}}^{4}|k|^{8}|t-s|^{4}
		\end{equation*}
		holds. Then, by the Burkholder-Davis-Gundy inequality,
		\begin{equation*}
			\mathbb{E}\left(\left|\frac{2 \sqrt{\nu}}{\|\theta\|_{\ell^{2}}} \sum_{l \in \mathbb{Z}_{0}^{2}} \theta_{l} \int_{s}^{t}\left\langle\omega_{N}(r), \sigma_{l} \cdot \nabla e_{k}\right\rangle \mathrm{d} B_{r}^{l}\right|^{4}\right) \leq C \frac{16 \nu^2}{\|\theta\|^{4}_{\ell^{2}}} \mathbb{E}\left(\left|\sum_{l \in \mathbb{Z}_{0}^{2}} \theta_{l}^{2} \int_{s}^{t}\left\langle\omega_{N}(r), \sigma_{l} \cdot \nabla e_{k}\right\rangle^{2} \mathrm{d} r\right|^{2}\right).
		\end{equation*}
		Noting that $\left\{\sigma_{l}\right\}_{l \in \mathbb{Z}_{0}^{2}}$ is an incomplete orthonormal family of $L^{2}(\mathbb{T}^2,\mathbb{R}^2)$ and using inequality (\ref{finiteL2ineq}), we have
		\begin{equation*}
			\begin{aligned}
				\sum_{l \in \mathbb{Z}_{0}^{2}} \theta_{l}^{2}\left\langle\omega_{N}(r), \sigma_{l} \cdot \nabla e_{k}\right\rangle^{2} & \leq\|\theta\|_{\ell^{\infty}}^{2} \sum_{l \in \mathbb{Z}_{0}^{2}}\left\langle\omega_{N}(r) \nabla e_{k}, \sigma_{l}\right\rangle^{2} \\
				& \leq\|\theta\|_{\ell^{\infty}}^{2}\left\|\omega_{N}(r) \nabla e_{k}\right\|_{L^{2}}^{2} \\
				& \leq C\|\theta\|_{\ell^{\infty}}^{2}|k|^{2}\left\|\omega_{0}\right\|_{L^{2}}^{2}.
			\end{aligned}
		\end{equation*}
		Therefore,
		\begin{equation*}
			\begin{aligned}
				\mathbb{E}\left(\left|\frac{2 \sqrt{\nu}}{\|\theta\|_{\ell^{2}}} \sum_{l \in \mathbb{Z}_{0}^{2}} \theta_{l} \int_{s}^{t}\left\langle\omega_{N}(r), \sigma_{l} \cdot \nabla e_{k}\right\rangle d B_{r}^{l}\right|^{4}\right) & \leq C^{\prime} \frac{16 \nu^2}{\|\theta\|^{4}_{\ell^{2}}}\|\theta\|_{\ell^{\infty}}^{4}|k|^{4}\left\|\omega_{0}\right\|_{L^{2}}^{4}|t-s|^{2} \\
				& \leq C|k|^{4}|t-s|^{2}.
			\end{aligned}
		\end{equation*}
		
		Combining the above estimates with (\ref{le1}), we finally get the desired inequality (\ref{four}).
	\end{proof}
\end{lemma}
By Cauchy's inequality and Lemma \ref{lemma1},
\begin{equation*}
	\begin{aligned}
		\mathbb{E}\left(\left\|\omega_{N}(t)-\omega_{N}(s)\right\|_{H^{-\beta}}^{4}\right) &=\mathbb{E}\left(\sum_{k \in \mathbb{Z}_{0}^{2}} \frac{\left\langle\omega_{N}(t)-\omega_{N}(s), e_{k}\right\rangle^{2}}{|k|^{2\beta}}\right)^{2} \\
		& \leq\left[\sum_{k \in \mathbb{Z}_{0}^{2}} \frac{1}{|k|^{2\beta}}\right]\left[\sum_{k \in \mathbb{Z}_{0}^{2}} \frac{\mathbb{E}\left(\left\langle\omega_{N}(t)-\omega_{N}(s), e_{k}\right\rangle^{4}\right)}{|k|^{2\beta}}\right] \\
		& \leq C|t-s|^{2} \sum_{k \in \mathbb{Z}_{0}^{2}} \frac{1}{|k|^{2\beta-8}} \leq C^{\prime}|t-s|^{2},
	\end{aligned}
\end{equation*}
where the last step is due to $\beta>5$. Thus,
\begin{equation*}
	\sup _{N \geq 1}\mathbb{E} \int_{0}^{T} \int_{0}^{T} \frac{\left\|\omega_{N}(t)-\omega_{N}(s)\right\|^4_{H^{-\beta}}}{|t-s|^{7 / 3}} \mathrm{d} t \mathrm{d} s \leq C.
\end{equation*}
We need to consider $\{\omega_N(\cdot)\}_{N\geq1}$
together with the laws of the sequence of Brownian motions. For this purpose, we endow $\mathbb{R}^{\mathbb{Z}_{0}^{2}}$ with the metric
\begin{align*}
	d_{\infty}(a, b)=\sum_{k \in \mathbb{Z}_{0}^{2}} \frac{\left|a_{k}-b_{k}\right| \wedge 1}{2^{|k|}}, \quad a, b \in \mathbb{R}^{\mathbb{Z}_{0}^{2}}.
\end{align*}
The metric space $\left(\mathbb{R}^{\mathbb{Z}_{0}}, d_{\infty}(\cdot, \cdot)\right)$ is separable and complete (see \cite{billingsley2013convergence}, Example 1.2). The distance of function space $C\left([0, T], \mathbb{R}^{\mathbb{Z}_{0}^{2}}\right)$ is given by
\begin{align*}
	d(w, \hat{w})=\sup _{t \in[0, T]} d_{\infty}(w(t), \hat{w}(t)), \quad w, \hat{w} \in C\left([0, T], \mathbb{R}^{\mathbb{Z}_{0}^{2}}\right),
\end{align*}
which makes $C\left([0, T], \mathbb{R}^{\mathbb{Z}_{0}^{2}}\right)$ a Polish space. We regard the sequence of Brownian motions $B_t:=\left\{\left(B_{t}^{k}\right)_{0 \leq t \leq T}: k \in \mathbb{Z}_{0}^{2}\right\}$ as a random variable with value in $C\left([0, T], \mathbb{R}^{\mathbb{Z}_{0}^{2}}\right)$. For any $N\geq 1$, denote by $P_N$ the joint law of $\left(\omega_N(\cdot), B_{\cdot}^{N}\right)$ on
\begin{equation*}
	\mathcal{X} \times \mathcal{Y}:=C\left([0, T] ; H^{-\delta}\right) \times C\left([0, T], \mathbb{R}^{\mathbb{Z}_{0}^{2}}\right).
\end{equation*}
Because the marginal laws are tight on $\mathcal{X}$ and $\mathcal{Y}$ respectively, $\{P_N\}_{N\geq1}$ is tight on $\mathcal{X} \times \mathcal{Y}$. The Prohorov's theorem implies that there exists a subsequence $\{N_i\}_{i\geq 1}$, such that $\{P_{N_i}\}_{i\geq 1}$ converges weakly as $i\rightarrow\infty$ to some probability measure $P$ on $\mathcal{X} \times \mathcal{Y}$ . By Skorokhod's representation theorem, there exists a family of stochastic processes $\left(\tilde{\omega}_{N_{i}}(\cdot), \tilde{B}^{N_{i}}\right)$ on some new probability space  $(\tilde{\Omega}, \tilde{\mathcal{F}}, \tilde{\mathbb{P}})$, such that
\begin{itemize}
	\item[(i)] $\left(\tilde{\omega}_{N_{i}}(\cdot), \tilde{B}^{N_{i}}\right)$ has the same law as $\left(\omega_{N_{i}}(\cdot), B^{N_{i}}\right)$ for any $i \geq 1$;
	\item[(ii)] $\tilde{\mathbb{P}}\text{-a.s.}$, $\left(\tilde{\omega}_{N_{i}}(\cdot), \tilde{B}^{N_{i}}\right)$ converges to the limit $(\tilde{\omega}(\cdot), \tilde{B})$ in the topology of $\mathcal{X} \times \mathcal{Y}$.
\end{itemize}
By (\ref{finiteL2ineq}), the following inequality holds:
\begin{align}\label{newl2ineq}
	\sup _{t \in[0, T]}\left\|\tilde{\omega}_{N_{i}}(t)\right\|_{L^{2}} \leq\left\|\omega_{0}\right\|_{L^{2}} \quad \tilde{\mathbb{P}} \text {-a.s. }
\end{align}
Denote by $\tilde{u}_{N_{i}}=K_0* \tilde{\omega}_{N_{i}}$ and $\tilde{u}=K_0* \tilde{\omega}$, which are the velocity fields defined on the new probability space $(\tilde{\Omega}, \tilde{\mathcal{F}}, \tilde{\mathbb{P}})$. By the above discussion, we know that $\tilde{\mathbb{P}}$-a.s., $\{\tilde{\omega}_{N_{i}}(\cdot)\}_{i\geq1}$ converges strongly to $\tilde{\omega}(\cdot)$ in $C\left([0, T] ; H^{-\delta}(\mathbb{T}^2)\right)$, which implies that
$\tilde{\mathbb{P}}$-a.s., $\{\tilde{u}_{N_{i}}(\cdot)\}_{i\geq1}$ converges strongly to $\tilde{u}(\cdot)$ in $C\left([0, T] ; H^{-\delta}(\mathbb{T}^2,\mathbb{R}^2)\right)$.
The candidate $\tilde{\omega}$ has the following properties: 
\begin{lemma}\label{trajectories}
	The process $\tilde{\omega}$ has $\tilde{\mathbb{P}}$-a.s. weakly continuous trajectories in $L^{2}$ and satisfies
	\begin{align}\label{boundness}
		\sup _{t \in[0, T]}\|\tilde{\omega}(t)\|_{L^{2}} \leq\left\|\omega_{0}\right\|_{L^{2}}  \quad \tilde{\mathbb{P}}\text{-a.s.}
	\end{align}
\end{lemma}
The proof of this lemma can be found in \cite{FGL} as Lemma 3.5.

Finally, we prove the convergence of the following equation as $i\rightarrow\infty$:
\begin{align}\label{Nisolution}
	\begin{aligned}
		\left\langle\tilde{\omega}_{N_{i}}(t), \phi\right\rangle=&\left\langle\omega_{N_{i}}(0), \phi\right\rangle+\int_{0}^{t}\left\langle\tilde{\omega}_{N_{i}}(s), \tilde{u}_{N_{i}}(s) \cdot \nabla\left(\Pi_{N_{i}} \phi\right)\right\rangle \mathrm{d} s \\
		&+\nu \int_{0}^{t}\left\langle\tilde{\omega}_{N_{i}}(s), \Delta \phi\right\rangle \mathrm{d} s-\frac{2 \sqrt{\nu}}{\|\theta\|_{\ell^{2}}} \sum_{k \in \mathbb{Z}_{0}^{2}} \theta_{k} \int_{0}^{t}\left\langle\tilde{\omega}_{N_{i}}(s), \sigma_{k} \cdot \nabla\left(\Pi_{N_{i}} \phi\right)\right\rangle \mathrm{d} \tilde{B}_{s}^{N_{i}, k},
	\end{aligned}
\end{align}
where $\phi\in C^{\infty}(\mathbb{T}^2)$ and $t \in[0, T]$.
All the terms are considered as real valued stochastic processes. It is a direct corollary that the left-hand side of (\ref{Nisolution}) and the first and third terms on the right-hand side of that converge $\tilde{\mathbb{P}}$-a.s.. Also, the proof of convergence of the stochastic integral term is standard (cf. Proof of Theorem 2.2 in \cite{FGL} or Proof of Theorem 2.1 in \cite{LS}). It is sufficient to prove the convergence of the nonlinear term which follows the idea presented by \cite{Resnick}.

First, we show that the nonlinear term can be expressed by the commutator as we already defined in (\ref{commutator}).  Recall that $\psi=-\Lambda^{-1}\omega$, $u=\nabla^{\perp}\psi$, for $\psi\in H^2(\T^2)$:
\begin{equation*}
	\begin{aligned}
		\int_{\mathbb{T}^{2}} \omega u \cdot \nabla \phi \mathrm{d} x=& \int_{\mathbb{T}^{2}} (-\Lambda \psi)\nabla^{\perp} \psi \cdot \nabla \phi \mathrm{d} x\\
		=&-\frac{1}{2}\int_{\mathbb{T}^{2}}  (\nabla^{\perp}(\psi\Lambda\psi)-\psi\nabla^{\perp}(\Lambda\psi))  \cdot \nabla \phi  \mathrm{d} x - \frac{1}{2}\int_{\mathbb{T}^{2}} (\Lambda \psi) \nabla^{\perp} \psi \cdot \nabla \phi \mathrm{d} x\\
		=&-\frac{1}{2}\int_{\mathbb{T}^{2}}  \nabla^{\perp}(\psi\Lambda\psi)   \cdot\nabla \phi \mathrm{d} x+\frac{1}{2}\int_{\mathbb{T}^{2}}  \psi\Lambda(\nabla^{\perp}\psi)  \cdot \nabla \phi  \mathrm{d} x
		-\frac{1}{2}\int_{\mathbb{T}^{2}} (\Lambda \psi)\nabla \phi  \cdot \nabla^{\perp} \psi \mathrm{d} x.
	\end{aligned}
\end{equation*}
The last equality is due to the commutativity of $\nabla^{\perp}$ and $\Lambda$. It is obvious that the first term equals to zero, and by the symmetric of the operator $\Lambda$, we have
\begin{equation}\label{nonlinearterm}
	\begin{aligned}
		\int_{\mathbb{T}^{2}} \omega u \cdot \nabla \phi \mathrm{d} x=&\frac{1}{2} \int_{\mathbb{T}^{2}} \left(\Lambda(\psi\nabla \phi )-(\Lambda \psi) \nabla \phi\right) \cdot \nabla^{\perp} \psi \mathrm{d} x\\=&\frac{1}{2} \int_{\mathbb{T}^{2}} u \cdot [\Lambda, \nabla \phi] \psi \mathrm{d} x.
	\end{aligned}
\end{equation}
Here, due to the commutator estimate (\ref{commutatorestimate}), it is enough to require $\psi\in H^1(\T^2)$, i.e. $\omega\in L^2(\T^2)$. Let $\tilde{\psi}=-\Lambda^{-1} \tilde{\omega}$, $\tilde{\psi}_{N_i}=-\Lambda^{-1} \tilde{\omega}_{N_i}$. By (\ref{nonlinearterm}),
\begin{align*}
	&\int_{0}^{t}\left\langle\tilde{\omega}, \tilde{u}\cdot\nabla \phi \right\rangle\mathrm{d}s-\int_{0}^{t}\left\langle \tilde{\omega}_{N_{i}}, \tilde{u}_{N_{i}}\cdot \nabla \Pi_{N_i} \phi \right\rangle\mathrm{d}s\\
	=&\frac{1}{2} \int_{0}^{t}\int_{{\mathbb{T}}^{2}}\left(\tilde{u}-\tilde{u}_{N_{i}}\right)\cdot[\Lambda, \nabla \phi] \tilde{\psi}\mathrm{d} x\mathrm{d}s
	+\frac{1}{2}\int_{0}^{t}\int_{{\mathbb{T}}^{2}}\tilde{u}_{N_{i}}\cdot\left[\Lambda, \nabla\left(\phi-\Pi_{N_i} \phi\right)\right] \tilde{\psi}\mathrm{d} x\mathrm{d}s\\
	&+\frac{1}{2}\int_{0}^{t}\int_{{\mathbb{T}}^{2}}\tilde{u}_{N_{i}}\cdot\left[\Lambda, \nabla \Pi_{N_i} \phi\right]\left(\tilde{\psi}-\tilde{\psi}_{N_i}\right) \mathrm{d} x\mathrm{d}s.
\end{align*}
$\tilde{u}_{N_i}$ is bounded, $\tilde{\psi}_{N_i}$ converges strongly to $\tilde{\psi}$ in $C\left([0, T] ; H^{1-\delta}\right)$ and $\Pi_{N_i} \phi$ converges strongly to $\phi$ in $H^{3+\epsilon}(\mathbb{T}^{2})$. Because of (\ref{commutatorestimate}), the second and third terms converge to 0 as $i\rightarrow\infty$. Since $[\Lambda, \nabla \phi] \tilde{\psi}$ is bounded in $L^2(\mathbb{T}^{2},\mathbb{R}^{2})$. To prove the convergence of the first term, it is enough to prove $\tilde{u}_{N_i}$ converges weakly to $\tilde{u}$ in $L^{2}\left(0, T ;L^2(\mathbb{T}^{2},\mathbb{R}^{2})\right)$. Equivalently, we need to prove the following lemma.
\begin{lemma}\label{weakL2}
	$\tilde{\mathbb{P}}\text{-a.s.}$, the process $\tilde{\omega}_{N_i}$ converges weakly to $\tilde{\omega}$ in $L^{2}\left(0, T ;L^2(\mathbb{T}^{2})\right)$.
	\begin{proof}
		Since $\{\tilde{\omega}_{N_i}\}_{i\geq1}$ converges strongly to $\tilde{\omega}$ in $C\left([0, T] ;H^{-\delta}(\mathbb{T}^{2})\right)$. Also, $\tilde{\omega}_{N_i}$ and $\tilde{\omega}$ both have $L^2$ trajectories. For any $\phi\in H^{\delta}$,
		\begin{equation}\label{hdelta}
			\begin{aligned}
				\sup _{t\in[0,T]}\left|\int_{\mathbb{T}^{2}}\left(\tilde{\omega}_{N_i}(t)-\tilde{\omega}(t)\right)\phi\d x\right|\leq&	\sup _{t\in[0,T]}\left\|\tilde{\omega}_{N_i}(t)-\tilde{\omega}(t)\right\|_{H^{-\delta}}\left\|\phi\right\|_{H^{\delta}}\\& =\left\|\tilde{\omega}_{N_i}-\tilde{\omega}\right\|_{C\left([0, T] ; H^{-\delta}(\mathbb{T}^2)\right)}\left\|\phi\right\|_{H^{\delta}}
			\end{aligned}
		\end{equation}
		converges to 0 as $i\rightarrow\infty$. Fixing $\phi^{\prime}\in L^2$, for any $\epsilon>0$, there exists $\phi\in H^{\delta}$, such that $\|\phi-\phi^{\prime}\|_{L^2}<\epsilon$, then
		\begin{align*}
			\sup _{t\in[0,T]}\left|\int_{\mathbb{T}^{2}}\left(\tilde{\omega}_{N_i}(t)-\tilde{\omega}(t)\right)\phi^{\prime} \mathrm{d} x\right| \leq \sup _{t\in[0,T]}\left|\int_{\mathbb{T}^{2}}\left(\tilde{\omega}_{N_i}(t)-\tilde{\omega}(t)\right)(\phi^{\prime}-\phi)\mathrm{d} x\right|+\sup _{t\in[0,T]}\left|\int_{\mathbb{T}^{2}}\left(\tilde{\omega}_{N_i}(t)-\tilde{\omega}(t)\right)\phi \mathrm{d} x\right|.
		\end{align*}
		Therefore, (\ref{hdelta}) implies that
		\begin{align*}
			\varlimsup_{i\rightarrow\infty}\sup _{t\in[0,T]}\left|\int_{\mathbb{T}^{2}}\left(\tilde{\omega}_{N_i}(t)-\tilde{\omega}(t)\right)\phi^{\prime} \mathrm{d} x\right| \leq 2\|\omega_{0}\|_{L^2}\|\phi^{\prime}-\phi\|_{L^2}<2\epsilon\|\omega_{0}\|_{L^2}.
		\end{align*}
		As a result, $\tilde{\mathbb{P}}\text{-a.s.}$, $\tilde{\omega}_{N_i}$ converges weakly to $\tilde{\omega}$ in $L^{2}\left(0, T ;L^2(\mathbb{T}^{2})\right)$.
	\end{proof}
\end{lemma}
Then, $\tilde{\mathbb{P}}\text{-a.s.}$,  for any $t\in\left[0,T\right]$.
\begin{align*}
	\int_{0}^{t}\left\langle\tilde{\omega}_{N_{i}}(s), \tilde{u}_{N_{i}}(s) \cdot \nabla\left(\Pi_{N_{i}} \phi\right)\right\rangle \mathrm{d} s \longrightarrow \int_{0}^{t}\left\langle\tilde{\omega}(s), \tilde{u}(s) \cdot \nabla\phi\right\rangle \mathrm{d} s \quad \text{as} \quad i\rightarrow\infty.
\end{align*}

Therefore, letting $i \rightarrow \infty$ in (\ref{Nisolution}), we obtain for all $t \in[0, T]$,
\begin{equation}\begin{aligned}
		\langle\tilde{\omega}(t), \phi\rangle
		=&\langle\omega_0, \phi\rangle+\int_{0}^{t}\langle\tilde{\omega}(s), \tilde{u}(s) \cdot \nabla \phi\rangle \mathrm{d} s+\nu \int_{0}^{t}\langle\tilde{\omega}(s), \Delta \phi\rangle \mathrm{d} s \\
		&-\frac{2 \sqrt{\nu}}{\|\theta\|_{\ell^{2}}} \sum_{k \in \mathbb{Z}_{0}^{2}} \theta_{k} \int_{0}^{t}\left\langle\tilde{\omega}(s), \sigma_{k} \cdot \nabla \phi\right\rangle \mathrm{d} \tilde{B}_{s}^{k}. 
\end{aligned}\end{equation}
This finishes the proof of Theorem \ref{main}.

\subsection{The Proof of Scaling Limit Theorem \ref{scalingtheoremSQG}}

\hspace{1.5em}Similar to (\ref{solution}), the solution to equation (\ref{SSQG}) is understood as follows: for any $\phi \in C^{\infty}\left(\mathbb{T}^{2}\right)$ and $t \in$ $[0, T]$,
\begin{equation}\label{solutionN}
	\begin{aligned}
		\left\langle\omega^{N}(t), \phi\right\rangle=&\left\langle\omega_{0}^{N}, \phi\right\rangle+\int_{0}^{t}\left\langle\omega^{N}(s), u^{N}(s) \cdot \nabla \phi\right\rangle \mathrm{d} s+\nu \int_{0}^{t}\left\langle\omega^{N}(s), \Delta \phi\right\rangle \mathrm{d} s \\
		&-\frac{2 \sqrt{\nu}}{\|\theta^{N}\|_{\ell^{2}}} \sum_{k \in \mathbb{Z}_{0}^{2}} \theta_{k}^{N} \int_{0}^{t}\left\langle\omega^{N}(s), \sigma_{k} \cdot \nabla \phi\right\rangle \mathrm{d} B_{s}^{k}.
	\end{aligned}
\end{equation}
By Theorem \ref{main}, there exists a weak solution $\omega^{N}$ to (\ref{solutionN}) satisfying
\begin{align}\label{finiteL2ineqN}
	\sup _{t \in[0, T]}\left\|\omega^{N}(t)\right\|_{L^{2}} \leq\left\|\omega_{0}^{N}\right\|_{L^{2}}\leq C \quad \mathbb{P} \text {-a.s.}
\end{align}
where the uniform bound of $\{\omega_0^N\}$ is obtained by the assumption in Theorem \ref{scalingtheoremSQG}. Although the process $\omega^{N}$ might be defined on different probability spaces, for simplicity, we do not distinguish the notations $\Omega, \mathbb{P}, \mathbb{E},$ etc. Let $Q^N$ denotes the law of $\omega^N$, we are going to show that $\left\{Q^{N}\right\}_{N \geq 1}$ is tight on $C\left([0,T],H^{-\delta}\right)$ for some $\delta\in(0,1)$. 
By the compact embedding assertion (i) of Theorem \ref{embedding}, it is sufficient to prove
\begin{align}\label{uniformestimateomegaN}
	\sup _{N \geq 1} \mathbb{E} \int_{0}^{T}\left\|\omega^{N}(t)\right\|_{L^2}^{p} \mathrm{d} t+\sup _{N \geq 1} \mathbb{E} \int_{0}^{T} \int_{0}^{T} \frac{\left\|\omega^{N}(t)-\omega^{N}(s)\right\|_{H^{-\beta}}^{4}}{|t-s|^{7 / 3}} \mathrm{d} t \mathrm{d} s<\infty.
\end{align}
\begin{lemma}\label{lemma1N}
	There is a constant $C=C(\nu,T,\left\|\omega_{0}\right\|_{L^2})$ such that for any $N \geq 1$ and $0 \leq s<t \leq T$,
	\begin{align*}
		\mathbb{E}\left(\left\langle\omega^{N}(t)-\omega^{N}(s), e_{k}\right\rangle^{4}\right) \leq C|k|^{8}|t-s|^{2}, \quad \text { for all } k \in \mathbb{Z}_{0}^{2}.
	\end{align*}
\end{lemma}
\begin{proof}
	For any fixed $k$, by (\ref{solutionN}), we have
	\begin{equation}\label{le1N}
		\begin{aligned}
			\left\langle\omega^{N}(t)-\omega^{N}(s), e_{k}\right\rangle=& \int_{s}^{t}\left\langle\omega^{N}(r), u^{N}(r) \cdot \nabla e_{k}\right\rangle \mathrm{d} r+\nu \int_{s}^{t}\left\langle\omega^{N}(r), \Delta e_{k}\right\rangle \mathrm{d} r \\
			&-\frac{2 \sqrt{\nu}}{\|\theta^{N}\|_{\ell^{2}}} \sum_{l \in \mathbb{Z}_{0}^{2}} \theta^{N}_{l} \int_{s}^{t}\left\langle\omega^{N}(r), \sigma_{l} \cdot \nabla e_{k}\right\rangle \mathrm{d} B_{r}^{l}.
		\end{aligned}
	\end{equation}
	Similarly to calculations in the proof of Lemma \ref{lemma1}, we can obtain the estimate of the first and the second terms on the right-hand side. By the Burkholder-Davis-Gundy inequality,
	\begin{equation*}
		\begin{aligned}
			\mathbb{E}\left(\left|\frac{2 \sqrt{\nu}}{\|\theta^{N}\|_{\ell^{2}}} \sum_{l \in \mathbb{Z}_{0}^{2}} \theta^N_{l} \int_{s}^{t}\left\langle\omega^{N}(r), \sigma_{l} \cdot \nabla e_{k}\right\rangle d B_{r}^{l}\right|^{4}\right) & \leq C \frac{16 \nu^2}{\|\theta^{N}\|^{4}_{\ell^{2}}}\|\theta^N\|_{\ell^{\infty}}^{4}|k|^{4}\left\|\omega_{0}^N\right\|_{L^{2}}^{4}|t-s|^{2}  .
		\end{aligned}
	\end{equation*}
	We obtain from (\ref{theta-N}) that the sequence $\{\frac{\left\|\theta^{N}\right\|_{\ell^{\infty}}}{\|\theta^{N}\|_{\ell^{2}}}\}_{N\geq1}$ is bounded and notice that $\|\omega^N_0\|_{L^2}$ has uniform bound. Therefore, the above inequality can be controlled by $ C^{\prime}|k|^{4}|t-s|^{2}$. Combining the estimate of three terms on the right-hand side, we finishes the proof of the desired inequality.
\end{proof}
Using Lemma \ref{lemma1N}, we can prove the uniform estimate \eqref{uniformestimateomegaN} by repeating the calculations below Lemma \ref{lemma1}. Therefore, the laws $\left\{Q^{N}\right\}_{N \geq 1}$ is tight on $C\left([0,T],H^{-\delta}\right)$ for some $\delta\in(0,1)$. The proof of Theorem \ref{scalingtheoremSQG} is similar to the proof of Theorem \ref{main} and we will show some key details.
\begin{proof}
	 By Prohorov's theorem and Skorokhod's representation theorem, there is a family of processes $\left(\hat{\omega}^{N_{i}}, \hat{B}^{N_{i}}\right)$ defined on some new probability space $(\hat{\Omega}, \hat{\mathcal{F}}, \hat{\mathbb{P}})$ such that
	\begin{itemize}
		\item[(i)] $\left(\hat{\omega}^{N_{i}}, \hat{B}^{N_{i}}\right)$ has the same law as $\left(\omega^{N_{i}}, B^{N_{i}}\right)$ for any $i \geq 1$;
		\item[(ii)]$\hat{\mathbb{P}}$-a.s., $\left(\hat{\omega}^{N_{i}}, \hat{B}^{N_{i}}\right)$ converges to
		$(\hat{\omega}, \hat{B})$ in the topology of $C\left([0, T] ; H^{-\delta}\right) \times C\left([0, T], \mathbb{R}^{\mathbb{Z}_{0}^{2}}\right)$.
	\end{itemize}
	
	Thus, equation (\ref{solutionN}) implies that for any $\phi \in C^{\infty}\left(\mathbb{T}^{2}\right)$,  for all $t \in[0, T]$, the following  equation holds $\hat{\mathbb{P}}$-a.s.:
	\begin{equation}\label{solutionNi}
		\begin{aligned}
			\left\langle\hat{\omega}^{N_{i}}(t), \phi\right\rangle
			=&\left\langle\omega^{N_i}_0,\phi\right\rangle+\int_{0}^{t}\left\langle\hat{\omega}^{N_{i}}(s), \hat{u}^{N_{i}}(s) \cdot \nabla \phi\right\rangle \mathrm{d} s+\nu \int_{0}^{t}\left\langle\hat{\omega}^{N_{i}}(s), \Delta \phi\right\rangle \mathrm{d} s \\
			&-\frac{2 \sqrt{\nu}}{\|\theta^{N_i}\|_{\ell^{2}}} \sum_{k \in \mathbb{Z}_{0}^{2}} \theta_{k}^{N_{i}} \int_{0}^{t}\left\langle\hat{\omega}^{N_{i}}(s), \sigma_{k} \cdot \nabla \phi\right\rangle \mathrm{d} \hat{B}_{s}^{N_{i}, k},
		\end{aligned}
	\end{equation}
	where $\hat{u}^{N_{i}}=K_{0} * \hat{\omega}^{N_{i}}$ are the velocity fields on the new probability space $(\hat{\Omega}, \hat{\mathcal{F}}, \hat{\mathbb{P}})$; moreover, a priori estimate holds:
	\begin{equation}\label{L2ineqNi}
		\sup _{t \in[0, T]}\left\|\hat{\omega}^{N_{i}}(t	)\right\|_{L^{2}} \leq\left\|\omega_{0}^{N_i}\right\|_{L^{2}}\leq C \quad \hat{\mathbb{P}} \text {-a.s.}
	\end{equation}
	The convergence of $\left\langle\omega_{0}^{N_i}, \phi\right\rangle$ is due to the assumption in Theorem \ref{scalingtheoremSQG}; the convergence of the nonlinear term can be treated in the same way as in the proof Theorem \ref{main}; and the convergence of the other terms in the first line of (\ref{solutionNi}) is easy to show. We only need to deal with the convergence of stochastic integral term. Actually, by It\^{o} isometry,
	\begin{equation*}\begin{aligned}
			\hat{\mathbb{E}}\left[\left(\frac{2 \sqrt{\nu}}{\|\theta^{N_i}\|_{\ell^{2}}}\sum_{k \in \mathbb{Z}_{0}^{2}} \theta_{k}^{N_{i}} \int_{0}^{t}\left\langle\hat{\omega}^{N_{i}}(s), \sigma_{k} \cdot \nabla \phi\right\rangle \mathrm{d} \hat{B}_{s}^{N_{i}, k}\right)^{2}\right] \\
			=\frac{4\nu}{\|\theta^{N_i}\|^2_{\ell^{2}}}\hat{\mathbb{E}}\left[\sum_{k \in \mathbb{Z}_{0}^{2}}\left(\theta_{k}^{N_{i}}\right)^{2} \int_{0}^{t}\left\langle\hat{\omega}^{N_{i}}(s), \sigma_{k} \cdot \nabla \phi\right\rangle^{2} \mathrm{d} s\right] \\
			\leq \frac{4\nu}{\|\theta^{N_i}\|^2_{\ell^{2}}}\left\|\theta^{N_{i}}\right\|_{\ell^{\infty}}^{2} \hat{\mathbb{E}} \int_{0}^{t} \sum_{k \in \mathbb{Z}_{0}^{2}}\left\langle\hat{\omega}^{N_{i}}(s), \sigma_{k} \cdot \nabla \phi\right\rangle^{2} \mathrm{d} s.
	\end{aligned}\end{equation*}
	By (\ref{L2ineqNi}) and using again the fact that $\{\sigma_{k}\}_{k\in\Z_0^2}$ is an incomplete orthonomal family in $L^2(\T^2,\R^2)$, we have
	\begin{equation*}
		\sum_{k \in \mathbb{Z}_{0}^{2}}\left\langle\hat{\omega}^{N_{i}}(s), \sigma_{k} \cdot \nabla \phi\right\rangle^{2} \leq\left\|\hat{\omega}^{N_{i}}(s) \nabla \phi\right\|_{L^{2}}^{2} \leq\|\nabla \phi\|_{\infty}^{2}\left\|\omega_{0}^{N_i}\right\|_{L^{2}}^{2}\leq C\|\nabla \phi\|_{\infty}^{2} \quad \hat{\mathbb{P}}\text{-a.s.}
	\end{equation*}
	Combining this fact with the condition (\ref{theta-N}), we conclude that, as $i \rightarrow \infty$ equation (\ref{solutionNi}) converges to
	\begin{equation*}
		\left\langle\hat{\omega}(t), \phi\right\rangle=\left\langle\omega_{0}, \phi\right\rangle+\int_{0}^{t}\left\langle\hat{\omega}(s), \hat{u}(s) \cdot \nabla \phi\right\rangle \mathrm{d} s+\nu \int_{0}^{t}\left\langle\hat{\omega}(s), \Delta \phi\right\rangle \mathrm{d} s.
	\end{equation*}
	This equation holds for all $\phi \in C^{\infty}\left(\mathbb{T}^{2}\right)$ and $t \in[0, T] .$ That is, $\hat{\omega}$ solves the dissipative SQG equation (\ref{dissipativeSQG}) in the weak sense. Because this equation admits a unique solution with initial data $\omega_{0} \in L^{2}$ (Theorem \ref{existenceanduniquenessforSQG}), we draw a conclusion that the whole sequence of equations (\ref{SSQG}) converges to the limit equation (\ref{dissipativeSQG}).
\end{proof}
Moreover, let us present a stronger result about solution to equation \eqref{SSQG}; moreover, given $\theta^N\in\ell^2$, we denote by
\begin{center}
	$\mathcal{C}_{\theta^{N}}(\omega_{0})$= collection of laws of weak solutions $\omega$ to (\ref{SSQG}) with initial data $\omega_{0}$.
\end{center}
Denote the space $\mathcal{X}=C\left([0, T], H^{-\delta}\left(\mathbb{T}^{2}\right)\right)$ endowed with the usual norm $\|\cdot\|_{\mathcal{X}}$ and denote by $\Phi_{\cdot}(\omega_0)$ the unique solution to deterministic dissipative SQG equation (\ref{dissipativeSQG}) with initial data $\omega_0\in L(\mathbb{T}^{2})$. Analogous to Theorem 1.2 in \cite{LS}, we generalize the discussion to SQG equation.
\begin{theorem}\label{SQGconvergenceinProb}
	For any $R>0$ and any $\epsilon>0$, we have
	\begin{equation}
		\lim_{N \rightarrow \infty} \sup _{\left\|\omega_{0}\right\|_{L^{2}} \leq R} \sup_{Q \in \mathcal{C}_{\theta^{N}}(\omega_{0})} Q\left(\left\{\varphi \in \mathcal{X}:\left\|\varphi-\Phi _{\cdot}(\omega_{0})\right\|_{\mathcal{X}}>\varepsilon\right\}\right)=0.
	\end{equation}
\end{theorem}
\begin{proof}
	We argue by contradiction. Suppose that there exists an $\epsilon_0>0$ such that
	\begin{equation}
		\lim_{N \rightarrow \infty} \sup _{\left\|\omega_{0}\right\|_{L^{2}} \leq R} \sup_{Q \in \mathcal{C}_{\theta^N}(\omega_{0})} Q\left(\left\{\varphi \in \mathcal{X}:\left\|\varphi-\Phi _{\cdot}(\omega_{0})\right\|_{\mathcal{X}}>\varepsilon_0\right\}\right)>0.
	\end{equation}
	Then, we can choose a subsequence $\left\{\omega^{N_{i}}_0\right\}_{i \geq 1}\in L^{2}\left(\mathbb{T}^{2}\right)$ with $\displaystyle\sup _{i \geq 1}\left\|\omega_{0}^{N_{i}}\right\|_{L^{2}} \leq R$ and $Q^{N_{i}} \in \mathcal{C}_{\theta^{N_{i}}}(\omega_{0}^{N_{i}})$. And for $\varepsilon_{0}$ small enough, we have
	\begin{equation}\label{contradiction1}
		Q^{N_{i}}\left(\left\{\varphi \in \mathcal{X}:\left\|\varphi-\Phi _{\cdot}\left(\omega_{0}^{N_{i}}\right)\right\|_{\mathcal{X}}>\varepsilon_{0}\right\}\right) \geq \varepsilon_{0}, \quad i \geq 1.
	\end{equation}
	First, since $\omega_{0}^{N_i}$ is uniformly bounded, up to a subsequence, we can assume that $\omega_{0}^{N_{i}}$
	converges weakly in $L^{2}\left(\mathbb{T}^{2}\right)$ to some $\omega_{0} .$ For any $i \geq 1,$ let $\omega^{N_{i}}$ be a weak solution to (\ref{SSQG}) in the sense of Definition \ref{weaksolutionofSQG} with the initial data $\omega_{0}^{N_{i}}$. Using again the boundedness of the family $\left\{\omega_{0}^{N_{i}}\right\}_{i \geq 1}$ and applying Theorem \ref{main}, we have
	\begin{equation*}
		\sup_{i \geq 1}\sup_{t\in[0,T]}\left\|\omega^{N_{i}}\right\|_{L^{2}}  \leq C_0<\infty;
	\end{equation*}
	moreover, $\omega^{N_{i}}$ satisfies equations (\ref{solutionNi}). Therefore, we can show that,  up to a further subsequence, $\omega^{N_{i}}$ converges weakly to the unique solution $\Phi _{\cdot}(\omega_{0})$ of the deterministic system (\ref{dissipativeSQG}) with initial data $\omega_{0}$, by repeating the arguments in the proof of Theorem \ref{scalingtheoremSQG}. Also we can conclude that $\omega^{N_{i}}$ converges also in probability to $\Phi_{\cdot}(\omega_{0})$. Thus,
	\begin{equation}\label{implies1}
		\lim _{i \rightarrow \infty} Q^{N_{i}}\left(\left\{\varphi \in \mathcal{X}:\left\|\varphi-\Phi .\left(\omega_{0}\right)\right\|_{\mathcal{X}}>\varepsilon_{0}\right\}\right)=0.
	\end{equation}
	Then, for any $i \geq 1,$ recall that $\Phi_{\cdot}\left(\omega_{0}^{N_{i}}\right)$ is the unique solution to the deterministic equation (\ref{dissipativeSQG}) with initial data $\omega_{0}^{N_{i}}$. By the equations in (\ref{dissipativeSQG}) and the weak convergence of $\omega_{0}^{N_{i}}$ to $\omega_{0}$, we find that, up to a subsequence,
	$\Phi_{\cdot}\left(\omega_{0}^{N_{i}}\right)$ converges to the limit $\Phi_{\cdot}(\omega_{0})$ in
	the topology of $\mathcal{X}$, which is the solution to the equation (\ref{dissipativeSQG}) with initial data $\omega_{0}$. The contradiction to (\ref{contradiction1})  occurs.
\end{proof}
As mentioned in the introduction, to be more precise, we show the following remark.
\begin{remark}\label{distanceremark}
	The laws $Q^N$ in $\mathcal{C}_{\theta^{N}}(\omega_{0})$ can be regarded as probability measure on $\mathcal{X}$. We denote $\rho_{\mathcal{X}}(\cdot,\cdot)$ as distance on $\mathcal{X}$ which metrizes the weak convergence of probability measures. It is a corollary of Theorem \ref{SQGconvergenceinProb} that 
	\begin{equation}
		\sup _{Q_{N}, Q_{N}^{\prime} \in \mathcal{C}_{\theta^{N}}\left(\omega_{0}\right)} \rho_{\mathcal{X}}\left(Q_{N}, Q_{N}^{\prime}\right) \rightarrow 0 \quad \text { as }\quad  N \rightarrow \infty.
	\end{equation}
This result shows that the solution to \eqref{SSQG} are approximately uniqueness.
\end{remark}
\subsection{The Uniqueness of Solution to Dissipative SQG Equation}

\hspace{1.5em}For the dissipative SQG, the global existence of solution with initial data $\omega_0\in L^2(\mathbb{T}^2)$ is studied by Resnick in \cite{Resnick} and the uniqueness in this sense is studied by Constantin and Wu in \cite{constantin1999behavior}. In \cite{kiselev2007global}, a stronger result can be proven: assume the initial data is $C^{\infty}$ and replace the dissipative term $\nu\Delta\omega$ by $\nu\Lambda\omega$, then SQG equation admits a unique global smooth solution.
\begin{theorem}\label{existenceanduniquenessforSQG}
	Let $T>0$ be arbitrary. Then for any initial data $\omega_0\in L^2$, there exists a unique weak solution to
	\begin{align}\label{dissipativeSQG2}
		\partial_{t}\omega+u\cdot\nabla\omega=\nu\Delta\omega,
	\end{align}	
	which satisfies
	\begin{align*}
		\omega \in L^{\infty}\left(0, T ; L^{2}\right) \cap L^{2}\left(0, T ; H^{1}\right).
	\end{align*}
\end{theorem}
To prove the uniqueness, we first introduce the following space
\begin{equation}\label{H}
	\mathcal{H}=\left\{f \in L^{2}\left(0, T ; H^{1}\left(\mathbb{T}^{2}\right)\right): \partial_{t} f \in L^{2}\left(0, T ; H^{-1}\left(\mathbb{T}^{2}\right)\right)\right\},
\end{equation}
where the derivative of time is understood in the distributional sense. $\mathcal{H}$ is continuously embedded into $C\left([0, T], L^{2}\left(\mathbb{T}^{2}\right)\right)$ and compactly embedded in $L^{2}\left([0, T], H^{m}\left(\mathbb{T}^{2}\right)\right)$ for any $m<-1$ (cf. Lemma 2.1.5 in \cite{kuksin2012mathematics}). Also, by (2.10) of \cite{kuksin2012mathematics} for any $f \in \mathcal{H}$ and $t \in[0, T]$:
\begin{equation}\label{derivatieoftime}
	\int_{0}^{t}\left\langle f, \partial_{s} f\right\rangle \mathrm{d} s=\frac{1}{2}\left(\left\|f_{t}\right\|_{L^{2}}^{2}-\left\|f_{0}\right\|_{L^{2}}^{2}\right).
\end{equation}
(\ref{derivatieoftime}) implies that $t \mapsto\left\|f_{t}\right\|_{L^{2}}^{2}$ is an absolutely continuous function and
$\frac{\mathrm{d}}{\mathrm{d} t}\left\|f_{t}\right\|_{L^{2}}^{2}=2\left\langle f, \partial_{t} f\right\rangle$. The original equation (\ref{dissipativeSQG2}) and a priori estimate imply that $\omega\in\mathcal{H}$.

\begin{proof}
	Let $\omega_A,\omega_B$ be two solutions to the equation \eqref{dissipativeSQG2} with the same initial data $\omega_0$. Denote by $u_{A}$ and $u_{B}$ the corresponding velocities respectively. The difference $\bar{\omega}:=\omega_A-\omega_B$ between two solutions $\omega_A$ and $\omega_B$ satisfies
	\begin{align}\label{difference}
		\partial_{t}\bar{\omega}+\bar{u}\cdot\nabla\omega_A+u_B\cdot\nabla\bar{\omega}-\nu\Delta\bar{\omega}=0.
	\end{align}
 where $\bar{u}=u_A-u_B$. We obtain from \eqref{derivatieoftime} that
	\begin{equation*}
		\frac{1}{2}\frac{\mathrm{d}}{\mathrm{d}t}\|\bar{\omega}\|^{2}_{L^2}+\nu\|\nabla\bar{\omega}\|^{2}_{L^{2}}=-\< \bar{u} \cdot \nabla \omega_A,\bar{\omega}\>,
	\end{equation*}
where we have used the fact that $u_{B}$ is divergence-free.	For the right-hand side, by the Sobolev embedding $H^{\frac{1}{2}}(\T^2) \hookrightarrow L^{4}(\T^2)$, we have
\begin{equation}\label{uomega1}
	\begin{aligned}
	\left| \< \bar{u} \cdot \nabla \omega_A,\bar{\omega}\>\right|
	\leq&\|\bar{u}\|_{L^4}\|\nabla\omega_A\|_{L^2}\|\bar{\omega}\|_{L^4}
	\\\leq&C\|\bar{u}\|_{H^{\frac{1}{2}}}\|\nabla\omega_A\|_{L^2}\|\bar{\omega}\|_{H^{\frac{1}{2}}}
	\\\leq&C\|\nabla\omega_A\|_{L^2}\|\bar{\omega}\|^2_{H^{\frac{1}{2}}},
	\end{aligned}
\end{equation}
where the last inequality is due to the same regularity of $\bar{u}$ and $\bar{\omega}$. Using interpolation inequality $\|\bar{\omega}\|_{H^{\frac{1}{2}}}\leq C\|\bar{\omega}\|^{\frac{1}{2}}_{L^{2}}\|\bar{\omega}\|^{\frac{1}{2}}_{H^{1}}$, we have
\begin{equation}\label{uomega2}
	\begin{aligned}
	\left| \< \bar{u} \cdot \nabla \omega_A,\bar{\omega}\>\right|
	\leq&C\|\nabla\omega_A\|_{L^2}\|\bar{\omega}\|_{L^{2}}\|\bar{\omega}\|_{H^{1}}
	\\\leq&C\|\nabla\omega_A\|_{L^2}\|\bar{\omega}\|_{L^{2}}\|\nabla\bar{\omega}\|_{L^2}
	\\\leq&\epsilon\|\nabla\bar{\omega}\|^2_{L^2}+C_{\epsilon}\|\nabla\omega_A\|^2_{L^2}\|\bar{\omega}\|^2_{L^2}.
	\end{aligned}
\end{equation}
We take $\epsilon=\nu$, then 
	\begin{align*}
		\frac{1}{2}\frac{\mathrm{d}}{\mathrm{d}t}\|\bar{\omega}\|^{2}_{L^2}\leq C_{\nu}\|\nabla\omega_A\|^{2}_{L^2}\|\bar{\omega}\|^{2}_{L^2}.
	\end{align*}
   Since $\omega_A\in L^2(0,T;H^1)$,	by Gronwall's inequality, it implies that $\left\|\bar{\omega}\right\|_{L^{2}}^{2}=0$ for all $t \in[0, T]$. Therefore, $\omega_{A}=\omega_{B}$, and the uniqueness of solution holds.
\end{proof}

\section{The Scaling Limit of Stochastic Boussinesq Equations}
\hspace{1.5em}We divide this section into three subsections: the proof of existence of weak solution to critical Boussinesq equations is given in subsection 4.1; the next subsection 4.2 is devoted to prove main scaling theorem of critical Boussinesq equations i.e. Theorem \ref{scalingresultforBoussinesq}; the purpose of subsection 4.3 is to show the uniqueness of solution to deterministic viscous Boussinesq equations.
\subsection{Existence of the Weak Solutions}
\hspace{1.5em}We first explain the meaning of weak solutions to (\ref{itoboussinesq}).
\begin{definition}\label{weaksolutionofboussinesq}
	We say that equations (\ref{itoboussinesq}) have a weak solution if there exists a filtered probability space $\left(\Xi, \mathcal{F}, \mathcal{F}_{t}, \mathbb{P}\right)$, a sequence of independent $\mathcal{F}_{t}$-Brownian motions $\left\{B^{k}\right\}_{k \in \mathbb{Z}_{0}^{2}}$ and $\mathcal{F}_{t}$-progressively measurable processes $\xi,\omega \in L^{2}\left(\Xi, L^{2}\left(0, T ; L^{2}\right)\right)$ with $\mathbb{P}$-a.s. weakly continuous trajectories such that for any $\phi \in C^{\infty}\left(\mathbb{T}^{2}\right),$ the following equalities holds $\mathbb{P}$-a.s. for all $t \in[0, T]$,
	\begin{equation}\begin{aligned}\label{boussinesqsolutionpart1}
			\left\langle\xi(t), \phi\right\rangle=&\left\langle\xi_{0}, \phi\right\rangle+\int_{0}^{t}\left[\left\langle\xi(s), u(s) \cdot \nabla \phi\right\rangle+(\kappa+\nu)\left\langle\xi(s), \Delta \phi\right\rangle\right] \mathrm{d} s \\
			&-\frac{2 \sqrt{\nu}}{\|\theta\|_{\ell^{2}}} \sum_{k \in \mathbb{Z}_{0}^{2}} \theta_{k} \int_{0}^{t}\left\langle\xi(s), \sigma_{k} \cdot \nabla \phi\right\rangle \mathrm{d} B_{s}^{k},
	\end{aligned}\end{equation}
	and
	\begin{equation}\begin{aligned}\label{boussinesqsolutionpart2}
			\left\langle\omega(t), \phi\right\rangle
			=&\left\langle\omega_{0}, \phi\right\rangle
			+\int_{0}^{t}\left[\left\langle\omega(s), u(s) \cdot \nabla \phi\right\rangle-\left\langle\xi(s), \partial_{1} \phi\right\rangle+\nu\left\langle\omega(s), \Delta \phi\right\rangle\right] \mathrm{d} s \\
			&-\frac{2 \sqrt{\nu}}{\|\theta\|_{\ell^{2}}} \sum_{k \in \mathbb{Z}_{0}^{2}} \theta_{k} \int_{0}^{t}\left\langle\omega(s), \sigma_{k} \cdot \nabla \phi\right\rangle \mathrm{d} B_{s}^{k}.
	\end{aligned}\end{equation}
\end{definition}
Similarly to Remark \ref{isometry}, the stochastic integrals in \eqref{boussinesqsolutionpart1} and \eqref{boussinesqsolutionpart2} make sense. The following theorem is the existence results of weak solutions to (\ref{itoboussinesq}).
\begin{theorem}\label{existenceBoussinesq}
	For any $\xi_0,\omega_{0}\in L^2(\mathbb{T}^2)$, the stochastic equations (\ref{itoboussinesq}) admits a weak solution satisfying
	\begin{equation*}
		\|\xi\|_{L^{\infty}\left(L^{2}\right)} \vee\|\nabla \xi\|_{L^{2}\left(L^{2}\right)} \leq C_{\kappa}\left\|\xi_{0}\right\|_{L^{2}}, \quad\|\omega\|_{L^{\infty}\left(L^{2}\right)} \leq\left\|\omega_{0}\right\|_{L^{2}}+C_{\kappa, T}\left\|\xi_{0}\right\|_{L^{2}}\quad \mathbb{P}\text{-a.s.}
	\end{equation*}	
	where we employ more compact notations as $\|\cdot\|_{L^{\infty}\left(L^{2}\right)}=\|\cdot\|_{L^{\infty}\left(0, T ; L^{2}\left(\mathbb{T}^{2}\right)\right)}$ and $\|\cdot\|_{L^{2}\left(L^{2}\right)}=\\\|\cdot\|_{L^{2}\left(0, T ; L^{2}\left(\mathbb{T}^{2}\right)\right)}$.
\end{theorem}
The rest of this subsection is devoted to the proof of Theorem \ref{existenceBoussinesq}, which makes use of the Galerkin approximation and compactness method as in the proof of Theorem \ref{main}.

First, we will look for an approximate solution $\left(\xi_N(\cdot),\omega_{N}(\cdot)\right)$. Let $H_N=\mathrm{span}\{e_k:k\in\mathbb{Z}^2_0,|k|\leq N\}$, and orthogonal projection $\Pi_{N}: L^{2}\left(\mathbb{T}^{2}\right) \rightarrow H_{N}$ and $\xi_N=\Pi_{N}\xi$, $\omega_N=\Pi_{N}\omega$. Operators $b_{N}$ and $G_{N}$ are defined the same as (\ref{bN}). In addition, we define
\begin{align*}
	h_N(\xi)=\Pi_{N}\left(\left(K_{0} * \Pi_{N} \omega\right) \cdot \nabla\left(\Pi_{N} \xi\right)\right).
\end{align*}
Properties (\ref{property}) still hold, and we have, 
\begin{align*}
	\<h_N(\xi_N),\xi_N\>=0, \quad \text{for all } \xi_N\in H_N.
\end{align*}
The finite dimensional version of (\ref{itoboussinesq}) is
\begin{equation}\label{finiteboussinesq}
	\left\{\begin{array}{l}
		\mathrm{d} \xi_{N}=-h_{N}\left(\xi_{N}(t)\right) \mathrm{d} t+(\kappa+\nu) \Delta \xi_{N}(t) \mathrm{d} t+\frac{2 \sqrt{\nu}}{\|\theta\|_{\ell^{2}}} \sum_{k \in \mathbb{Z}_{0}^{2}} \theta_{k} G_{N}^{k}\left(\xi_{N}(t)\right) \mathrm{d} B_{t}^{k}, \\
		\mathrm{d} \omega_{N}=-b_{N}\left(\omega_{N}(t)\right) \mathrm{d} t+\left(\partial_{1} \xi_{N}+\nu \Delta \xi_{N}\right) \mathrm{d} t+\frac{2 \sqrt{\nu}}{\|\theta\|_{\ell^{2}}} \sum_{k \in \mathbb{Z}_{0}^{2}} \theta_{k} G_{N}^{k}\left(\omega_{N}(t)\right) \mathrm{d} B_{t}^{k},
	\end{array}\right.
\end{equation}
with the initial data $\xi_{N}(0)=\Pi_{N} \xi_{0}$ and $\omega_{N}(0)=\Pi_{N} \omega_{0}$. It follows from standard SDE theory that there exists a unique local strong solution $\left(\xi_N(\cdot),\omega_{N}(\cdot)\right)$ under any initial conditions.

We now investigate a priori estimate of $\xi_N$ and $\omega_N$.
Similarly to the computations below (\ref{finitedimension}), we have
\begin{align*}
	\mathrm{d}\left\|\xi_{N}(t)\right\|_{L^{2}}^{2}\leq-2\kappa\left\|\nabla\xi_{N}(t)\right\|_{L^{2}}^{2}\mathrm{d}t.
\end{align*}
Thus,
\begin{equation}\label{estimatexiN}
	\left\|\xi_N(t)\right\|_{L^{2}}^{2}+2\kappa \int_{0}^{t}\left\|\nabla \xi_{N}(s)\right\|_{L^{2}}^{2} \mathrm{d} s \leq\left\|\xi_{N}(0)\right\|_{L^{2}}^{2}\leq\left\|\xi_0\right\|_{L^{2}}^{2}, 
\end{equation}
Similarly,
\begin{align*}
	\mathrm{d}\left\|\omega_{N}(t)\right\|_{L^{2}}^{2}&\leq 2\left\langle\omega_{N}(t),\partial_{1}\xi_{N}(t) \right\rangle \mathrm{d}t\leq2\|\omega_{N}(t)\|_{L^{2}}\left\|\partial_{1} \xi_{N}(t)\right\|_{L^{2}} \mathrm{d} t \\&\leq 2\|\omega_{N}(t)\|_{L^{2}}\|\nabla \xi_{N}(t)\|_{L^{2}} \mathrm{d} t.
\end{align*}
That is to say, $\mathrm{d}\left\|\omega_{N}(t)\right\|_{L^{2}}\leq\|\nabla \xi_{N}(t)\|_{L^{2}} \mathrm{d} t.$
Therefore,
\begin{equation}\label{estimateomegaN}
	\begin{aligned}
		\left\|\omega_{N}(t)\right\|_{L^{2}} &\leq\left\|\omega_{N}(0)\right\|_{L^{2}}+\int_{0}^{t}\left\|\nabla \xi_{N}(s)\right\|_{L^{2}} \mathrm{d} s \leq\left\|\omega_{N}(0)\right\|_{L^{2}}+\sqrt{T}\|\nabla \xi_{N}\|_{L^{2}\left(L^{2}\right)}\\&\leq\left\|\omega_{N}(0)\right\|_{L^{2}}+C_{\kappa,T}\| \xi_{N}(0)\|_{L^{2}}\\&\leq\left\|\omega_{0}\right\|_{L^{2}}+C_{\kappa, T}\left\|\xi_{0}\right\|_{L^{2}}.
	\end{aligned}
\end{equation}
By (\ref{estimatexiN}) and (\ref{estimateomegaN}), global existence of solution to (\ref{finiteboussinesq}) holds; moreover,
\begin{equation}\label{estimateofxiNandomegaN}
	\|\xi_N\|_{L^{\infty}\left(L^{2}\right)} \vee\|\nabla \xi_N\|_{L^{2}\left(L^{2}\right)} \leq C_{\kappa}\left\|\xi_{0}\right\|_{L^{2}}, \quad\|\omega_N\|_{L^{\infty}\left(L^{2}\right)} \leq\left\|\omega_{0}\right\|_{L^{2}}+C_{\kappa, T}\left\|\xi_{0}\right\|_{L^{2}}.
\end{equation}

Then, we will look for a candidate $(\tilde{\xi}(\cdot),\tilde{\omega}(\cdot),B_{\cdot})$ as a weak solution to equations (\ref{itoboussinesq}).
By the compact embedding Theorem \ref{embedding}, to show the laws of $\{\xi_N\}_{N\geq1}$ is tight on $C\left([0, T], H^{-\delta}\right)\cap L^{2}\left(0, T ; L^{2}\right)$ and the laws of $\{\omega_N\}_{N\geq1}$ is tight on $C\left([0, T], H^{-\delta}\right)$, it is sufficient to prove
\begin{equation}\label{1}
	\sup _{N \geq 1} \mathbb{E} \sup _{t\in[0,T] }\left\|\xi_{N}(t)\right\|_{L^{2}} +\sup _{N \geq 1} \mathbb{E} \int_{0}^{T} \int_{0}^{T} \frac{\left\|\xi_{N}(t)-\xi_{N}(s)\right\|_{H^{-\beta}}^{4}}{|t-s|^{7 / 3}} \mathrm{d} t \mathrm{d} s<\infty,
\end{equation}
\begin{equation}\label{2}
	\sup _{N \geq 1} \mathbb{E} \int_0^t \left\|\xi_{N}(t)\right\|_{H^{1}}^{2}\mathrm{d} t +\sup _{N \geq 1} \mathbb{E} \int_{0}^{T} \int_{0}^{T} \frac{\left\|\xi_{N}(t)-\xi_{N}(s)\right\|_{H^{-\beta}}^{2}}{|t-s|^{1+2\gamma}} \mathrm{d} t \mathrm{d} s<\infty,
\end{equation}
and
\begin{equation}\label{3}
	\sup _{N \geq 1} \mathbb{E} \sup _{t\in[0,T] }\left\|\omega_{N}(t)\right\|_{L^{2}} +\sup _{N \geq 1} \mathbb{E} \int_{0}^{T} \int_{0}^{T} \frac{\left\|\omega_{N}(t)-\omega_{N}(s)\right\|_{H^{-\beta}}^{4}}{|t-s|^{7 / 3}} \mathrm{d} t \mathrm{d} s<\infty.
\end{equation}
The first terms of the above three inequalities have been controlled by estimate (\ref{estimateofxiNandomegaN}) respectively. It remains to estimate the second terms on the left-hand side, similarly to Lemma \ref{lemma1}.
\begin{lemma}\label{lemma5}
	There is a constant $C=C(\kappa,\nu,T,\|\xi_0\|_{L^2},\|\omega_0\|_{L^2})$, such that for any $N \geq 1$ and $0 \leq s<t \leq T$,
	\begin{equation*}
		\mathbb{E}\left(\left\langle\xi_{N}(t)-\xi_{N}(s), e_{k}\right\rangle^{4}\right)\vee\mathbb{E}\left(\left\langle\omega_{N}(t)-\omega_{N}(s), e_{k}\right\rangle^{4}\right) \leq C|k|^{8}|t-s|^{2} \quad \text { for all } k \in \mathbb{Z}_{0}^{2}.
	\end{equation*}	
\end{lemma}
\begin{proof}
It is enough to consider $|k|\leq N$. By (\ref{finiteboussinesq}), we have
\begin{equation*}\begin{aligned}
		\left\langle\xi_{N}(t)-\xi_{N}(s), e_{k}\right\rangle=& \int_{s}^{t}\left\langle\xi_{N}(r), u_{N}(r) \cdot \nabla e_{k}\right\rangle \mathrm{d} r+(\kappa+\nu) \int_{s}^{t}\left\langle\xi_{N}(r), \Delta e_{k}\right\rangle \mathrm{d} r \\
		&-\frac{2 \sqrt{\nu}}{\|\theta\|_{\ell^{2}}} \sum_{l \in \mathbb{Z}_{0}^{2}} \theta_{l} \int_{s}^{t}\left\langle\xi_{N}(r), \sigma_{l} \cdot \nabla e_{k}\right\rangle \mathrm{d} B_{r}^{l},
\end{aligned}\end{equation*}
and	
\begin{equation*}\begin{aligned}
		\left\langle\omega_{N}(t)-\omega_{N}(s), e_{k}\right\rangle=& \int_{s}^{t}\left\langle\omega_{N}(r), u_{N}(r) \cdot \nabla e_{k}\right\rangle \mathrm{d} r+\nu \int_{s}^{t}\left\langle\omega_{N}(r), \Delta e_{k}\right\rangle \mathrm{d} r \\
		&-\int_{s}^{t}\left\langle\omega_N(r),\partial_{1}e_k\right\rangle\mathrm{d} r-\frac{2 \sqrt{\nu}}{\|\theta\|_{\ell^{2}}} \sum_{l \in \mathbb{Z}_{0}^{2}} \theta_{l} \int_{s}^{t}\left\langle\omega_{N}(r), \sigma_{l} \cdot \nabla e_{k}\right\rangle \mathrm{d} B_{r}^{l}.
\end{aligned}\end{equation*}
Compared to the proof of Lemma \ref{lemma1}, the only term remains to estimate is
\begin{equation*}\begin{aligned}
		\mathbb{E}\left(\left|\int_{s}^{t}\left\langle\omega_{N}(r), \partial_{1}e_{k}\right\rangle \mathrm{d} r\right|^{4}\right) &\leq C\left\|\omega_{0}\right\|_{L^{2}}^{4}\|\partial_{1}e_{k}\|^{4}_{L^2}|t-s|^{4}\\&\leq C\left\|\omega_{0}\right\|_{L^{2}}^{4}\|\nabla e_{k}\|^{4}_{L^2}|t-s|^{4}\\&\leq C\left\|\omega_{0}\right\|_{L^{2}}^{4}|k|^{4}|t-s|^{4}.
	\end{aligned}
\end{equation*}
Having Lemma \ref{lemma5} in hand, we can prove the uniform estimates in \eqref{1}-\eqref{3} by repeating the calculations below Lemma \ref{lemma1}.
\end{proof}
We denote $\eta_N$ as the joint law of the pair of processes  $(\xi_N,\omega_{N})$, $N\geq1$. By the discussion above, the family $\{\eta_N\}_{N\geq1}$ is tight on
\begin{equation}\label{mathcalz}
	\mathcal{Z}=\left(L^{2}\left(0, T; L^{2}\left(\mathbb{T}^{2}\right)\right) \cap C\left([0, T]; H^{-\delta}\left(\mathbb{T}^{2}\right)\right)\right) \times C\left([0, T]; H^{-\delta}\left(\mathbb{T}^{2}\right)\right).
\end{equation}
The rest argument is parallel to stochastic SQG equation. Denote by $Q_N$ the joint law of $(\xi_N(\cdot),\omega_{N}(\cdot),B_{\cdot})$ on
\begin{equation}
	\mathcal{Z}\times\mathcal{Y}=\left(L^{2}\left(0, T; L^{2}\left(\mathbb{T}^{2}\right)\right) \cap C\left([0, T]; H^{-\delta}\left(\mathbb{T}^{2}\right)\right)\right) \times C\left([0, T]; H^{-\delta}\left(\mathbb{T}^{2}\right)\right)\times C\left([0, T];\mathbb{R}^{\mathbb{Z}_{0}^{2}}\right).
\end{equation}
and $\left\{Q_{N}\right\}_{N \geq 1}$ is tight on $\mathcal{Z} \times \mathcal{Y}$. Skorokhod's representation theorem implies that there exists a new probability space $(\tilde{\Xi}, \tilde{\mathcal{F}}, \tilde{\mathbb{P}})$, a sequence of stochastic processes $(\tilde{\xi}_{N_i}(\cdot),\tilde{\omega}_{N_i}(\cdot),B^{N_i}_{\cdot})$ and a limit process $(\tilde{\xi}(\cdot),\tilde{\omega}(\cdot),B_{\cdot})$ defined on $(\tilde{\Xi}, \tilde{\mathcal{F}}, \tilde{\mathbb{P}})$, such that
\begin{itemize}
	\item [(i)] $\left(\tilde{\xi}_{N_{i}}, \tilde{\omega}_{N_{i}},\tilde{B}^{N_i}\right)$ has the same law as $\left(\xi_{N_{i}}, \omega_{N_{i}},B^{N_i}\right)$ for any $i\geq1$;
	\item [(ii)] $\tilde{\mathbb{P}}$-a.s., $\left(\tilde{\xi}_{N_{i}}, \tilde{\omega}_{N_{i}},\tilde{B}^{N_i}\right)$ converges to the limit process $(\tilde{\xi}, \tilde{\omega},\tilde{B})$ in the topology of
	\begin{align*}
		\left(L^{2}\left(0, T; L^{2}\left(\mathbb{T}^{2}\right)\right) \cap C\left([0, T]; H^{-\delta}\left(\mathbb{T}^{2}\right)\right)\right) \times C\left([0, T];H^{-\delta}\left(\mathbb{T}^{2}\right)\right)\times C\left([0, T];\mathbb{R}^{\mathbb{Z}_{0}^{2}}\right).
	\end{align*}
\end{itemize}
Denoted $\tilde{u}=K_{0} * \tilde{\omega}$, it implies that $\tilde{\mathbb{P}}$-a.s., $\tilde{u}_{N_{i}}(\cdot)$ converges strongly to $\tilde{u}(\cdot)$ $\operatorname{in} C\left([0, T] ; H^{-\delta}\left(\mathbb{T}^{2}, \mathbb{R}^{2}\right)\right)$. We also can prove that the limit processes $\tilde{\xi}$ and $\tilde{\omega}$ have $\tilde{\mathbb{P}}$-a.s. weakly continuous trajectories, and satisfy
\begin{align*}
\tilde{\mathbb{P}}\text{-a.s}, \quad\|\tilde{\xi}\|_{L^{\infty}\left(L^{2}\right)}\vee\|\nabla \tilde{\xi}\|_{L^{2}\left(L^{2}\right)} \leq\left\|\xi_{0}\right\|_{L^{2}} \quad \text { and } \quad\|\tilde{\omega}\|_{L^{\infty}\left(L^{2}\right)} \leq\left\|\omega_{0}\right\|_{L^{2}}+C_{\kappa, T}\left\|\xi_{0}\right\|_{L^{2}}.
\end{align*}

Finally, we prove the convergence of the following equations as $i\rightarrow\infty$:
\begin{equation}\label{Nixi}
	\begin{aligned}
		\left\langle\tilde{\xi}_{N_{i}}(t), \phi\right\rangle=&\left\langle\xi_{N_{i}}(0), \phi\right\rangle+\int_{0}^{t}\left\langle\tilde{\xi}_{N_{i}}(s), \tilde{u}_{N_{i}}(s) \cdot \nabla\left(\Pi_{N_{i}} \phi\right)\right\rangle \mathrm{d} s \\
		&+(\nu+\kappa) \int_{0}^{t}\left\langle\tilde{\xi}_{N_{i}}(s), \Delta \phi\right\rangle \mathrm{d} s-\frac{2 \sqrt{\nu}}{\|\theta\|_{\ell^{2}}} \sum_{k \in \mathbb{Z}_{0}^{2}} \theta_{k} \int_{0}^{t}\left\langle\tilde{\xi}_{N_{i}}(s), \sigma_{k} \cdot \nabla\left(\Pi_{N_{i}} \phi\right)\right\rangle \mathrm{d} \tilde{B}_{s}^{N_{i}, k},
	\end{aligned}
\end{equation}
and
\begin{equation}\label{Niomega}
	\begin{aligned}
		\left\langle\tilde{\omega}_{N_{i}}(t), \phi\right\rangle=&\left\langle\omega_{N_{i}}(0), \phi\right\rangle+\int_{0}^{t}\left\langle\tilde{\omega}_{N_{i}}(s), \tilde{u}_{N_{i}}(s) \cdot \nabla\left(\Pi_{N_{i}} \phi\right)\right\rangle \mathrm{d} s -\int_{0}^{t}\left\langle\tilde{\omega}_{N_i}(s),\partial_{1}\phi\right\rangle\mathrm{d} s\\
		&+\nu \int_{0}^{t}\left\langle\tilde{\omega}_{N_{i}}(s), \Delta \phi\right\rangle \mathrm{d} s-\frac{2 \sqrt{\nu}}{\|\theta\|_{\ell^{2}}} \sum_{k \in \mathbb{Z}_{0}^{2}} \theta_{k} \int_{0}^{t}\left\langle\tilde{\omega}_{N_{i}}(s), \sigma_{k} \cdot \nabla\left(\Pi_{N_{i}} \phi\right)\right\rangle \mathrm{d} \tilde{B}_{s}^{N_{i}, k},
	\end{aligned}
\end{equation}
where $\phi\in C^{\infty}(\mathbb{T}^2)$ and $t \in[0, T]$. The proof of the stochastic integral terms is standard. Since $\tilde{\xi}_{N_{i}}$ converges to $\tilde{\xi}$ in the topology of $L^{2}\left(0, T; L^{2}\left(\mathbb{T}^{2}\right)\right)$, it is easy to show that, $\tilde{\mathbb{P}}\text{-a.s.}$,
\begin{align*}
	\int_{0}^{t}\left\langle\tilde{\xi}_{N_{i}}(s), \tilde{u}_{N_{i}}(s) \cdot \nabla\left(\Pi_{N_{i}} \phi\right)\right\rangle \mathrm{d} s \longrightarrow \int_{0}^{t}\langle\tilde{\xi}(s), \tilde{u}(s) \cdot \nabla \phi\rangle \mathrm{d} s \quad \text{as} \quad i \rightarrow \infty.
\end{align*}
The nonlinear term in (\ref{Niomega}) can be treated in the same way as in Section 3.1. Let $\tilde{\psi}=-\Lambda^{-1} \tilde{\omega}, \tilde{\psi}_{N_{i}}=-\Lambda^{-1} \tilde{\omega}_{N_{i}}$. Then, $\tilde{\psi}_{N_{i}}$ converges strongly to $\tilde{\psi}$ in $C\left([0, T] ; H^{1-\delta}\right)$, and
\begin{align*}
	&\int_{0}^{t}\left\langle\tilde{\omega}, \tilde{u}\cdot\nabla \phi \right\rangle\mathrm{d}s-\int_{0}^{t}\left\langle \tilde{\omega}_{N_{i}}, \tilde{u}_{N_{i}}\cdot \nabla \Pi_{N_i} \phi \right\rangle\mathrm{d}s\\
	=&\frac{1}{2} \int_{0}^{t}\int_{{\mathbb{T}}^{2}}\left(\tilde{u}-\tilde{u}_{N_{i}}\right)\cdot[\Lambda, \nabla \phi] \tilde{\psi}\mathrm{d} x\mathrm{d}s
	+\frac{1}{2}\int_{0}^{t}\int_{{\mathbb{T}}^{2}}\tilde{u}_{N_{i}}\cdot\left[\Lambda, \nabla\left(\phi-\Pi_{N_i} \phi\right)\right] \tilde{\psi}\mathrm{d} x\mathrm{d}s\\
	&+\frac{1}{2}\int_{0}^{t}\int_{{\mathbb{T}}^{2}}\tilde{u}_{N_{i}}\cdot\left[\Lambda, \nabla \Pi_{N_i} \phi\right]\left(\tilde{\psi}-\tilde{\psi}_{N_i}\right) \mathrm{d} x\mathrm{d}s.
\end{align*}
Recalling the Lemma \ref{weakL2}, $\tilde{u}_{N_i}$ converges weakly to $\tilde{u}$ in $L^{2}\left(0, T ;L^2(\mathbb{T}^{2},\mathbb{R}^{2})\right)$, and the commutator estimate (\ref{commutatorestimate}) implies that for any $t \in[0, T]$, $\tilde{\mathbb{P}}\text{-a.s.}$,
\begin{align*}
	\int_{0}^{t}\left\langle\tilde{\omega}_{N_{i}}(s), \tilde{u}_{N_{i}}(s) \cdot \nabla\left(\Pi_{N_{i}} \phi\right)\right\rangle \mathrm{d}s \longrightarrow \int_{0}^{t}\langle\tilde{\omega}(s), \tilde{u}(s) \cdot \nabla \phi\rangle \mathrm{d} s \quad \text{as} \quad i \rightarrow \infty.
\end{align*}
The convergence of the other terms in (\ref{Nixi}) and (\ref{Niomega}) is obvious. Therefore, letting $i\rightarrow\infty$, we conclude that the pair $(\tilde{\xi},\tilde{\omega})$ satisfies the equations (\ref{boussinesqsolutionpart1}) and (\ref{boussinesqsolutionpart2}). This finishes the proof of existence of weak solution to stochastic Boussinesq equations (\ref{itoboussinesq}).

\subsection{The Proof of Scaling Limit Theorem \ref{scalingresultforBoussinesq}}
\begin{proof}By the a priori estimate in Theorem \ref{existenceBoussinesq}, for $N\geq1$,
	\begin{equation}\label{xiomegaestimate}
		\|\xi^N\|_{L^{\infty}\left(L^{2}\right)} \vee\|\nabla \xi^N\|_{L^{2}\left(L^{2}\right)} \leq\left(1 \vee \kappa^{-1 / 2}\right)\left\|\xi_{0}^N\right\|_{L^{2}}, \quad\|\omega^N\|_{L^{\infty}\left(L^{2}\right)} \leq\left\|\omega_{0}^N\right\|_{L^{2}}+C_{\kappa, T}\left\|\xi_{0}^N\right\|_{L^{2}}.
	\end{equation}	
	Similarly to the previous discussion, we can show the laws of $\{\xi^N\}_{N\geq1}$ are tight on $C\left([0, T], H^{-\delta}\right)\cap L^{2}\left(0, T ; L^{2}\right)$ and the laws of $\{\omega_N\}_{N\geq1}$ are tight on $C\left([0, T], H^{-\delta}\right)$. Consequently, Prohorov's theorem and Skorokhod's representation theorem imply that there exist a new probability space $(\hat{\Xi}, \hat{\mathcal{F}}, \hat{\mathbb{P}})$, a sequence of stochastic processes $(\hat{\xi}^{N_i}(\cdot),\hat{\omega}^{N_i}(\cdot),\hat{B}^{N_i}_{\cdot})$ and a limit process $(\hat{\xi}(\cdot),\hat{\omega}(\cdot),\hat{B}_{\cdot})$ defined on $(\hat{\Xi}, \hat{\mathcal{F}}, \hat{\mathbb{P}})$, such that
	\begin{itemize}
		\item[(i)] $\left(\hat{\xi}^{N_{i}}, \hat{\omega}^{N_{i}}, \hat{B}^{N_i}\right)$ has the same law as $\left(\xi^{N_{i}}, \omega^{N_{i}}, B^{N_i}\right)$ for any $i\geq1$;
		\item[(ii)] $\hat{\mathbb{P}}$-a.s., $\left(\hat{\xi}^{N_{i}}, \hat{\omega}^{N_{i}},\hat{B}^{N_i}\right)$ converges to the limit process $(\hat{\xi}, \hat{\omega},\hat{B})$ in the topology of
		\begin{align*}
			\left(L^{2}\left(0, T; L^{2}\left(\mathbb{T}^{2}\right)\right) \cap C\left([0, T]; H^{-\delta}\left(\mathbb{T}^{2}\right)\right)\right) \times C\left([0, T];H^{-\delta}\left(\mathbb{T}^{2}\right)\right)\times C\left([0, T];\mathbb{R}^{\mathbb{Z}_{0}^{2}}\right).
		\end{align*}
	\end{itemize}
	The assertion (i) implies that $\hat{\xi}^{N_i}$ and $\hat{\omega}^{N_i}$ satisfy the same equations as \eqref{Nixi} and \eqref{Niomega}; more precisely, for any $\phi \in C^{\infty}\left(\mathbb{T}^{2}\right)$, for all $t \in[0, T]$, $\hat{\mathbb{P}}$-a.s.
	\begin{equation}\label{boussinesqsolutionN1}
		\begin{aligned}
			\left\langle\hat{\xi}^{N_{i}}(t), \phi\right\rangle=&\left\langle\xi^{N_{i}}(0), \phi\right\rangle+\int_{0}^{t}\left\langle\hat{\xi}^{N_{i}}(s), \hat{u}^{N_{i}}(s) \cdot \nabla\phi\right\rangle \mathrm{d} s
			+(\nu+\kappa) \int_{0}^{t}\left\langle\hat{\xi}^{N_{i}}(s), \Delta \phi\right\rangle \mathrm{d} s\\&-\frac{2 \sqrt{\nu}}{\left\|\theta^{N_i}\right\|_{\ell^{2}}} \sum_{k \in \mathbb{Z}_{0}^{2}} \theta_{k}^{N_i} \int_{0}^{t}\left\langle\hat{\xi}^{N_{i}}(s), \sigma_{k} \cdot \nabla\phi\right\rangle \mathrm{d} \hat{B}_{s}^{N_{i}, k},
		\end{aligned}
	\end{equation}
	and
	\begin{equation}\label{boussinesqsolutionN2}
		\begin{aligned}
			\left\langle\hat{\omega}^{N_{i}}(t), \phi\right\rangle=&\left\langle\omega^{N_{i}}(0), \phi\right\rangle+\int_{0}^{t}\left\langle\hat{\omega}^{N_{i}}(s), \hat{u}^{N_{i}}(s) \cdot \nabla\phi\right\rangle \mathrm{d} s -\int_{0}^{t}\left\langle\hat{\omega}^{N_i}(s),\partial_{1}\phi\right\rangle\mathrm{d} s\\
			&+\nu \int_{0}^{t}\left\langle\hat{\omega}^{N_{i}}(s), \Delta \phi\right\rangle \mathrm{d} s-\frac{2 \sqrt{\nu}}{\left\|\theta^{N_i}\right\|_{\ell^{2}}} \sum_{k \in \mathbb{Z}_{0}^{2}} \theta_{k}^{N_i} \int_{0}^{t}\left\langle\hat{\omega}^{N_{i}}(s), \sigma_{k} \cdot \nabla\phi\right\rangle \mathrm{d} \hat{B}_{s}^{N_{i}, k}.
		\end{aligned}
	\end{equation}
Similarly to the calculations below \eqref{L2ineqNi}, we can show that the stochastic integral terms in \eqref{boussinesqsolutionN1} and \eqref{boussinesqsolutionN2} vanish as $i\rightarrow\infty$. Indeed, by the It\^{o} isometry,
	\begin{equation*}\begin{aligned}
			&\hat{\mathbb{E}}\left[\left(\frac{2 \sqrt{\nu}}{\|\theta^{N_{i}}\|_{\ell^{2}}} \sum_{k \in \mathbb{Z}_{0}^{2}} \theta_{k}^{N_{i}} \int_{0}^{t}\left\langle\hat{\xi}^{N_{i}}(s), \sigma_{k} \cdot \nabla \phi\right) \mathrm{d} \hat{B}_{s}^{N_{i}, k}\right)^{2}\right] \\
			=&\frac{4 \nu}{\|\theta^{N_{i}}\|^2_{\ell^{2}}} \hat{\mathbb{E}}\left[\sum_{k \in \mathbb{Z}_{0}^{2}}\left(\theta_{k}^{N_{i}}\right)^{2} \int_{0}^{t}\left\langle\hat{\xi}^{N_{i}}(s), \sigma_{k} \cdot \nabla \phi\right\rangle^{2} \mathrm{d} s\right] \\
			\leq& (\frac{4 \nu}{\|\theta^{N_{i}}\|^2_{\ell^{2}}}\left\|\theta^{N_{i}}\right\|_{\ell^{\infty}}^{2} \hat{\mathbb{E}} \int_{0}^{t} \sum_{k \in \mathbb{Z}_{0}^{2}}\left\langle\hat{\xi}^{N_{i}}(s), \sigma_{k} \cdot \nabla \phi\right\rangle^{2} \mathrm{d} s\\
			=&4\nu\frac{\left\|\theta^{N_{i}}\right\|_{\ell^{\infty}}^{2}}{\left\|\theta^{N_i}\right\|_{\ell^{2}}^{2}} \hat{\mathbb{E}} \int_{0}^{t} \sum_{k \in \mathbb{Z}_{0}^{2}}\left\langle\hat{\xi}^{N_{i}}(s), \sigma_{k} \cdot \nabla \phi\right\rangle^{2} \mathrm{d} s.
	\end{aligned}\end{equation*}
	By assertion (i), $\hat{\xi}^{N_i}$ fufills the first estimate in \eqref{xiomegaestimate} $\hat{\P}$-a.s.; therefore,
	\begin{equation*}
		\sum_{k \in \mathbb{Z}_{0}^{2}}\left\langle\hat{\xi}^{N_{i}}(s), \sigma_{k} \cdot \nabla \phi\right\rangle^{2} \leq\left\|\hat{\xi}^{N_{i}}(s) \nabla \phi\right\|_{L^{2}}^{2} \leq\|\nabla \phi\|_{\infty}^{2}\left\|\xi_{0}^{N_i}\right\|_{L^{2}}^{2}\leq C\|\nabla \phi\|_{\infty}^{2} \quad \hat{\mathbb{P}}\text{-a.s.}
	\end{equation*}
	Combining this fact with the condition (\ref{theta-N}), we conclude that, as $i \rightarrow \infty$, the stochastic integral terms in (\ref{boussinesqsolutionN1}) converge to 0. The analogous argument also implies that the stochastic integral terms in (\ref{boussinesqsolutionN2}) vanishes as $i\rightarrow\infty$. The proof of convergence of other terms is parallel to that in Theorem \ref{existenceBoussinesq}. Thus, we get the limit equations:
	\begin{equation}
		\begin{aligned}
			\left\langle\hat{\xi}(t), \phi\right\rangle=\left\langle\xi(0), \phi\right\rangle+\int_{0}^{t}\left\langle\hat{\xi}(s), \hat{u}(s) \cdot \nabla\phi\right\rangle \mathrm{d} s +(\nu+\kappa) \int_{0}^{t}\left\langle\hat{\xi}(s), \Delta \phi\right\rangle \mathrm{d} s,
		\end{aligned}
	\end{equation}
	and
	\begin{equation}
		\begin{aligned}
			\left\langle\hat{\omega}(t), \phi\right\rangle=\left\langle\omega(0), \phi\right\rangle+\int_{0}^{t}\left\langle\hat{\omega}(s), \hat{u}(s) \cdot \nabla\phi\right\rangle \mathrm{d} s -\int_{0}^{t}\left\langle\hat{\omega}(s),\partial_{1}\phi\right\rangle\mathrm{d} s+\nu \int_{0}^{t}\left\langle\hat{\omega}(s), \Delta \phi\right\rangle \mathrm{d} s.
		\end{aligned}
	\end{equation}
	That is to say, the limit $(\hat{\xi},\hat{\omega})$ is a weak solution to (\ref{deterministicviscousBoussinesq}) with initial data $(\xi_0,\omega_0)$.
\end{proof}
Analogous to the discussion of Theorem \ref{SQGconvergenceinProb}, the similar result can be obtained for the scaling limit of stochastic Boussinesq equations (cf. Theorem 1.2 in \cite{Luo20}). For $(\xi_0,\omega_0)\in(L(\mathbb{T}^2))^2$, define $\|(\xi_0,\omega_0)\|_{L^2}=\|\xi_0\|_{L^2}\vee\|\omega_0\|_{L^2}$; moreover, given $\theta^N\in \ell^2$, we denote by
\begin{center}
	$\mathcal{C}_{\theta^N}\left(\xi_{0}, \omega_{0}\right)$= collection of laws of weak solutions $(\xi, \omega)$ to (\ref{stratonovichboussinesq}) with initial data $\left(\xi_{0}, \omega_{0}\right)$.
\end{center}
Recall the space $\mathcal{Z}$ defined in \eqref{mathcalz}, 
endowed with the usual norm $\|\cdot\|_{\mathcal{Z}}$ of product space. Finally, we denote by $\Phi_{\cdot}\left(\xi_{0}, \omega_{0}\right)$ the unique solution to deterministic viscous Boussinesq equations (\ref{deterministicviscousBoussinesq}) with initial data $(\xi_0,\omega_0)\in(L^2(\mathbb{T}^{2}))^{2}$. We can prove the following result by repeating the arguments of Theorem \ref{SQGconvergenceinProb}: 
\begin{theorem}\label{BoussinesqconvergenceinProb}
	For any $R>0$ and any $\epsilon>0$, we have
	\begin{equation}
		\lim_{N \rightarrow \infty} \sup _{\left\|\left(\xi_{0}, \omega_{0}\right)\right\|_{L^{2}} \leq R} \sup_{Q \in \mathcal{C}_{\theta^N}\left(\xi_{0}, \omega_{0}\right)} Q\left(\left\{\varphi \in \mathcal{Z}:\left\|\varphi-\Phi _{\cdot}\left(\xi_{0}, \omega_{0}\right)\right\|_{\mathcal{Z}}>\varepsilon\right\}\right)=0.
	\end{equation}
\end{theorem}
We remark that $Q_N$ in $\mathcal{C}_{\theta^N}\left(\xi_{0}, \omega_{0}\right)$ can be viewed as a measure on $\mathcal{Z}$ with distance $\rho_{\mathcal{Z}}(\cdot,\cdot)$ metrizing its weak convergence. Similarly to Remark \ref{distanceremark}, since Theorem \ref{BoussinesqconvergenceinProb} implies that 
\begin{equation}
	\sup _{Q_{N}, Q_{N}^{\prime} \in \mathcal{C}_{\theta^ N}\left(\xi_{0}, \omega_{0}\right)} \rho_{\mathcal{Z}}\left(Q_{N}, Q_{N}^{\prime}\right) \rightarrow 0 \quad \text { as } N \rightarrow \infty,
\end{equation} 
it is reasonable to say that the solution to \eqref{stratonovichboussinesq} are approximately unique.
\subsection{The Uniqueness of Solution to Viscous Boussinesq equations}
\begin{theorem}
	Let $T>0$ be arbitrary. Then for every $\left(\xi_{0}, \omega_{0}\right) \in\left(L^{2}\left(\mathbb{T}^{2}\right)\right)^{2}$, there exists a unique weak solution to
	\begin{equation}\label{viscoussystem}
		\left\{\begin{array}{l}
			\partial_{t} \xi+u \cdot \nabla \xi=(\kappa+\nu) \Delta \xi, \\
			\partial_{t} \omega+u \cdot \nabla \omega=\partial_{1} \xi+\nu \Delta \omega, \\
			u=K_0 * \omega,
		\end{array}\right.
	\end{equation}
	which satisfies
	\begin{equation*}
		\xi, \omega \in L^{\infty}\left(0, T ; L^{2}\left(\mathbb{T}^{2}\right)\right) \cap L^{2}\left(0, T ; H^{1}\left(\mathbb{T}^{2}\right)\right).
	\end{equation*}
	Moreover, there exists a constant $C_{T}>0$ such that the solution satisfies the bound
	\begin{equation*}
		\|\xi\|_{L^{\infty}\left(L^{2}\right)} \vee\|\xi\|_{L^{2}\left(H^{1}\right)} \vee\|\omega\|_{L^{\infty}\left(L^{2}\right)} \vee\|\omega\|_{L^{2}\left(H^{1}\right)} \leq C_{T}\left(\left\|\omega_{0}\right\|_{L^{2}}+\left\|\xi_{0}\right\|_{L^{2}}\right).
	\end{equation*}
\end{theorem}
When the kernel $K_0$ is replaced by a more regular one, a similar proof is given in the appendix of \cite{Luo20}. Without loss of generality, we assume $\kappa=\nu=1$ for simplicity.

We first give some classical a priori estimates. By the divergence-free property of $u$, the first equation in (\ref{viscoussystem}) yields
\begin{equation}\label{estimatexi}
	\left\|\xi(t)\right\|_{L^{2}}^{2}+4 \int_{0}^{t}\left\|\nabla \xi(s)\right\|_{L^{2}}^{2} \mathrm{d} s=\left\|\xi_{0}\right\|_{L^{2}}^{2}, \quad t \in[0, T].
\end{equation}
Using the second equation in (\ref{viscoussystem}) we have
\begin{equation}\label{estimatedomega}
	\frac{\mathrm{d}}{\mathrm{d} t}\|\omega\|_{L^{2}}^{2}=2\left\langle\omega, \Delta \omega+\partial_{1} \xi\right\rangle \leq-2\|\nabla \omega\|_{L^{2}}^{2}+2\|\omega\|_{L^{2}}\left\|\partial_{1} \xi\right\|_{L^{2}}\leq2\|\omega\|_{L^{2}}\|\nabla \xi\|_{L^{2}}.
\end{equation}
It yields that
\begin{equation}\label{estimateprioriomega}
	\begin{aligned}
		\left\|\omega(t)\right\|_{L^{2}} & \leq\left\|\omega_{0}\right\|_{L^{2}}+\int_{0}^{t}\left\|\nabla \xi(s)\right\|_{L^{2}} \mathrm{d} s \\
		& \leq\left\|\omega_{0}\right\|_{L^{2}}+\sqrt{T}\|\nabla \xi\|_{L^{2}\left(L^{2}\right)} \leq(1 \vee \sqrt{T})\left(\left\|\omega_{0}\right\|_{L^{2}}+\left\|\xi_{0}\right\|_{L^{2}}\right),
	\end{aligned}
\end{equation}
where the last step is due to (\ref{estimatexi}). From (\ref{estimatedomega}) we also obtain
\begin{equation*}
	\begin{aligned}
		\left\|\omega(t)\right\|_{L^{2}}^{2}+2 \int_{0}^{t}\left\|\nabla \omega(s)\right\|_{L^{2}}^{2} \mathrm{d} s & \leq\left\|\omega_{0}\right\|_{L^{2}}^{2}+2 \int_{0}^{t}\left\|\omega(s)\right\|_{L^{2}}\left\|\partial_{1} \xi(s)\right\|_{L^{2}} \mathrm{d} s \\
		& \leq\left\|\omega_{0}\right\|_{L^{2}}^{2}+2\|\omega\|_{L^{2}\left(L^{2}\right)}\|\nabla \xi\|_{L^{2}\left(L^{2}\right)}.
	\end{aligned}
\end{equation*}
Combining this estimate with (\ref{estimatexi}) and (\ref{estimateprioriomega}), we have
\begin{equation}\label{estimateomega}
	\left\|\omega(t)\right\|_{L^{2}}^{2}+2 \int_{0}^{t}\left\|\nabla \omega(s)\right\|_{L^{2}}^{2} \mathrm{d} s \leq C^{\prime}\left(\left\|\omega_{0}\right\|_{L^{2}}^{2}+\left\|\xi_{0}\right\|_{L^{2}}^{2}\right).
\end{equation}
Gathering a priori estimates (\ref{estimatexi}) and (\ref{estimateomega}), we have
\begin{equation}\label{estimateinall}
	\|\xi\|_{L^{\infty}\left(L^{2}\right)} \vee\|\xi\|_{L^{2}\left(H^{1}\right)} \vee\|\omega\|_{L^{\infty}\left(L^{2}\right)} \vee\|\omega\|_{L^{2}\left(H^{1}\right)} \leq C_{T}\left(\left\|\omega_{0}\right\|_{L^{2}}+\left\|\xi_{0}\right\|_{L^{2}}\right),
\end{equation}
with some $C_T>0$. By \eqref{estimateinall}, the existence of solutions to the equations (\ref{viscoussystem}) can be obtained using the Galerkin approximation methods.

Recall (\ref{H}) for the definition of $\mathcal{H}$. The equations (\ref{viscoussystem}) and the estimate \eqref{estimateinall} imply that $\xi, \omega \in \mathcal{H}$. Let $\left(\xi^{A}, \omega^{A}\right)$ and $\left(\xi^{B}, \omega^{B}\right)$ be two solutions to the equations (\ref{viscoussystem}) with the same initial data $\left(\xi_{0}, \omega_{0}\right) ;$ denote by $u^{A}$ and $u^{B}$ the corresponding velocities respectively. Let $\bar{\xi}=\xi^{A}-\xi^{B}$, $\bar{\omega}=\omega^{A}-\omega^{B}$ and $\bar{u}=u^{A}-u^{B}$, then
\begin{equation}\label{xiomegaAB1}
	\partial_{t} \bar{\xi}-2\Delta \bar{\xi} =-\left(\bar{u} \cdot \nabla \xi^{A}+u^{B} \cdot \nabla \bar{\xi}\right) ,
\end{equation}
and
\begin{equation}\label{xiomegaAB2}	
		\partial_{t} \bar{\omega}-\Delta \bar{\omega} =\partial_{1} \bar{\xi}-\left(\bar{u} \cdot \nabla \omega^{A}+u^{B} \cdot \nabla \bar{\omega}\right).
\end{equation}
We obtain from \eqref{xiomegaAB1} and (\ref{derivatieoftime}) that
\begin{equation}\label{uniqueness}
	\begin{aligned}
		\frac{1}{2}\frac{\d}{\d t}\left\|\bar{\xi}\right\|_{L^{2}}^{2}+2\left\|\nabla \bar{\xi}\right\|_{L^{2}}^{2} &=-\left[\left\langle\bar{u} \cdot \nabla \xi^{A}, \bar{\xi}\right\rangle+\left\langle u^{B} \cdot \nabla \bar{\xi}, \bar{\xi}\right\rangle\right] \\&=-\left\langle\bar{u} \cdot \nabla \xi^{A}, \bar{\xi}\right\rangle,
	\end{aligned}
\end{equation}
where we have used the fact that $u^{B}$ is divergence-free. Similarly,  (\ref{xiomegaAB2}) yields
\begin{equation}\label{uniquenessomega}
	\frac{1}{2}\frac{\d}{\d t}\left\|\bar{\omega}\right\|_{L^{2}}^{2}+\left\|\nabla \bar{\omega}\right\|_{L^{2}}^{2} =-\left\langle\bar{u} \cdot \nabla \omega^{A}, \bar{\omega}\right\rangle +\left\langle\partial_{1} \bar{\xi}, \bar{\omega}\right\rangle .
\end{equation}

By H\"{o}lder's inequality and the Sobolev embedding $ H^{\frac{1}{2}}(\T^2) \hookrightarrow L^{4}(\T^2)$, the right-hand side of (\ref{uniqueness}) can be estimated as 
\begin{equation*}
	\begin{aligned}
	\left|\left\langle\bar{u} \cdot \nabla \xi^{A}, \bar{\xi}\right\rangle\right| \leq&\left\|\nabla \xi^{A}\right\|_{L^{2}}\left\|\bar{\xi}\right\|_{L^{4}}\left\|\bar{u}\right\|_{L^{4}} \\\leq& C\left\|\nabla\xi^{A}\right\|_{L^2}\left\|\bar{\xi}\right\|_{H^\frac{1}{2}}\left\|\bar{u}\right\|_{H^\frac{1}{2}}\\\leq& C\left\|\nabla\xi^{A}\right\|_{L^2}\left\|\bar{\xi}\right\|_{H^\frac{1}{2}}\left\|\bar{\omega}\right\|_{H^\frac{1}{2}},
	\end{aligned}
\end{equation*}
where the last inequality is due to the same regularity of $\bar{u}$ and $\bar{\omega}$. Using interpolation inequality, we have
\begin{equation*}
	\begin{aligned}
		\left|\left\langle\bar{u} \cdot \nabla \xi^{A}, \bar{\xi}\right\rangle\right| \leq& C\left\|\nabla\xi^{A}\right\|_{L^2}\left\|\bar{\omega}\right\|^{\frac{1}{2}}_{H^1}\left\|\bar{\omega}\right\|^{\frac{1}{2}}_{L^2}\left\|\bar{\xi}\right\|^{\frac{1}{2}}_{H^1}\left\|\bar{\xi}\right\|^{\frac{1}{2}}_{L^2}\\\leq& 
		C\left\|\nabla\bar{\omega}\right\|^{\frac{1}{2}}_{L^2}
		\left\|\nabla\bar{\xi}\right\|^{\frac{1}{2}}_{L^2}
			\left(\left\|\nabla\xi^{A}\right\|^{\frac{1}{2}}_{L^2}
	\left\|\bar{\omega}\right\|^{\frac{1}{2}}_{L^2}\right)
	\left(	\left\|\nabla\xi^{A}\right\|^{\frac{1}{2}}_{L^2}
			\left\|\bar{\xi}\right\|^{\frac{1}{2}}_{L^2} \right)
		\\\leq&\epsilon_{1}\left\|\nabla\bar{\omega}\right\|^{2}_{L^2}+
		\epsilon_{2}\left\|\nabla\bar{\xi}\right\|^{2}_{L^2}+C_{1}\left\|\nabla\xi^{A}\right\|^2_{L^2}\left\|\bar{\omega}\right\|^{2}_{L^2}
		+C_{2}\left\|\nabla\xi^{A}\right\|^2_{L^2}\left\|\bar{\xi}\right\|^{2}_{L^2},
	\end{aligned}
\end{equation*}
where $\epsilon_1$ and $\epsilon_{2}$ are small numbers will be determined later. Similarly to previous estimate \eqref{uomega1} and \eqref{uomega2}, for the right-hand side of (\ref{uniquenessomega}), we obtain
\begin{equation*}
	\left|\left\langle\bar{u} \cdot \nabla \omega^{A}, \bar{\omega}\right\rangle\right| \leq\epsilon_3\left\|\nabla \bar{\omega}\right\|_{L^{2}}^{2}+C_{3}\left\|\nabla\omega^{A}\right\|_{L^{2}}^{2}\left\|\bar{\omega}\right\|_{L^{2}}^{2},
\end{equation*}
and
\begin{equation*}
	\left|\left\langle\partial_{1} \bar{\xi}, \bar{\omega}\right\rangle\right| \leq\left\|\partial_{1} \bar{\xi}\right\|_{L^{2}}\left\|\bar{\omega}\right\|_{L^{2}} \leq\left\|\nabla \bar{\xi}\right\|_{L^{2}}\left\|\bar{\omega}\right\|_{L^{2}} \leq \epsilon_4\left\|\nabla \bar{\xi}\right\|_{L^{2}}^{2}+C_{4}\left\|\bar{\omega}\right\|_{L^{2}}^{2}.
\end{equation*}
\eqref{uniqueness} and \eqref{uniquenessomega} together yield that
\begin{equation*}
	\begin{aligned}
		&\frac{1}{2}\frac{\d}{\d t}\left\|\bar{\xi}\right\|_{L^{2}}^{2}+ \frac{1}{2}\frac{\d}{\d t}\left\|\bar{\omega}\right\|_{L^{2}}^{2}+2\left\|\nabla \bar{\xi}\right\|_{L^{2}}^{2} +\left\|\nabla \bar{\omega}\right\|_{L^{2}}^{2}
		\\\leq&\left|\left\langle\bar{u} \cdot \nabla \xi^{A}, \bar{\xi}\right\rangle\right|+	\left|\left\langle\bar{u} \cdot \nabla \omega^{A}, \bar{\omega}\right\rangle\right|+ \left|\left\langle\partial_{1} \bar{\xi}, \bar{\omega}\right\rangle\right|
		\\\leq&(\epsilon_2+\epsilon_4)\left\|\nabla \bar{\xi}\right\|_{L^{2}}^{2}+(\epsilon_1+\epsilon_3)\left\|\nabla \bar{\omega}\right\|_{L^{2}}^{2}\\&+\left(C_{1}\left\|\nabla\xi^{A}\right\|^2_{L^2}+C_{3}\left\|\nabla\omega^{A}\right\|_{L^{2}}^{2}+C_{4}\right)\left\|\bar{\omega}\right\|^{2}_{L^2}+C_{2}\left\|\nabla\xi^{A}\right\|^2_{L^2}\left\|\bar{\xi}\right\|^{2}_{L^2}.
	\end{aligned}
\end{equation*}
Take $\epsilon_2+\epsilon_4=2$ and $\epsilon_1+\epsilon_3=1$, we have
\begin{equation*}
	\begin{aligned}
		&\frac{1}{2}\frac{\d}{\d t}\left(\left\|\bar{\xi}\right\|_{L^{2}}^{2}+\left\|\bar{\omega}\right\|_{L^{2}}^{2}\right) 
		\leq C\left(\left\|\nabla\xi^{A}\right\|^2_{L^2}+\left\|\nabla\omega^{A}\right\|_{L^{2}}^{2}+1\right)\left(\left\|\bar{\xi}\right\|^{2}_{L^2}+\left\|\bar{\omega}\right\|^{2}_{L^2}\right).
	\end{aligned}
\end{equation*}
Since $\xi^A$ and $\omega^A$ are in $L^2(0,T;H^1)$, by Gronwall's
inequality, it implies that $\left\|\bar{\xi}\right\|_{L^{2}}^{2}+\left\|\bar{\omega}\right\|_{L^{2}}^{2}=0$ for all $t \in[0, T]$. Therefore, $\xi^{A}=\xi^{B}$ and $\omega^{A}=\omega^{B}$, and the uniqueness of solution holds.

\end{document}